\documentclass[a4paper]{article}
\usepackage{amsthm,amsmath,amssymb}
\usepackage{enumerate,enumitem}
\usepackage{graphicx}
\usepackage{xfrac}
\usepackage{ulem}
\usepackage{mathtools}
\usepackage{ wasysym }
\usepackage{tabularx,booktabs,multirow}
\usepackage{hyperref}
\newcolumntype{C}{>{\centering\arraybackslash}X}
\newcolumntype{D}{>{\centering\arraybackslash}X}

\DeclareGraphicsExtensions{.pdf,.png,.eps}
\graphicspath{{figs/}}

\newtheorem{theorem}{Theorem}
\newtheorem{lemma}[theorem]{Lemma}

\newtheorem{corollary}[theorem]{Corollary}

\newtheorem{conjecture}[theorem]{Conjecture}

\newtheorem*{claim*}{Claim}

\theoremstyle{remark}

\newcommand{\cB}{\ensuremath{\mathcal{B}}}
\newcommand{\cD}{\ensuremath{\mathcal{D}}}

\newcommand{\cG}{\ensuremath{\mathcal{G}}}
\newcommand{\cJ}{\ensuremath{\mathcal{J}}}
\newcommand{\cP}{\ensuremath{\mathcal{P}}}
\newcommand{\cS}{\ensuremath{\mathcal{S}}}

\newcommand{\N}{\ensuremath{\mathbb{N}}}
\newcommand{\cN}{\ensuremath{\mathcal{N}}}

\newcommand{\R}{\ensuremath{\mathbb{R}}}

\newcommand{\cV}{\ensuremath{\mathcal{V}}}

\newcommand{\cX}{\ensuremath{\mathcal{X}}}
\newcommand{\cY}{\ensuremath{\mathcal{Y}}}
\newcommand{\cZ}{\ensuremath{\mathcal{Z}}}

\newcommand{\PP}{\ensuremath{\mathcal{P}}}
\newcommand{\QQ}{\ensuremath{\mathcal{Q}}}

\newcommand{\CC}{\ensuremath{\mathcal{C}}}

\newcommand{\zero}{\ensuremath{\texttt{0}}}
\newcommand{\one}{\ensuremath{\texttt{1}}}

\newcommand{\cLa}{\mathrel{\reflectbox{\rotatebox[origin=c]{180}{$\cV$}}}}

\newcommand{\subsetneql}{\ensuremath{\subset}}

\usepackage[usenames,dvipsnames]{xcolor}

\begin{document}

\title{Poset Ramsey number $R(P,Q_n)$. III. Chain Compositions and Antichains}
\author{Christian Winter\footnote{Karlsruhe Institute of Technology,  Karlsruhe, Germany E-mail: \textit{christian.winter@kit.edu}}}

\maketitle

\begin{abstract}
An induced subposet $(P_2,\le_2)$ of a poset $(P_1,\le_1)$ is a subset of $P_1$ such that for every two $X,Y\in P_2$, $X\le_2 Y$ if and only if $X\le_1 Y$.
The Boolean lattice $Q_n$ of dimension $n$ is the poset consisting of all subsets of $\{1,\dots,n\}$ ordered by inclusion.

Given two posets $P_1$ and $P_2$ the poset Ramsey number $R(P_1,P_2)$ is the smallest integer $N$ such that in any blue/red coloring of the elements of $Q_N$ there is either a monochromatically blue induced subposet isomorphic to $P_1$ or a monochromatically red induced subposet isomorphic to $P_2$.

We provide upper bounds on $R(P,Q_n)$ for two classes of $P$: parallel compositions of chains, i.e.\ posets consisting of disjoint chains which are pairwise element-wise incomparable, 
as well as subdivided $Q_2$, which are posets obtained from two parallel chains by adding a common minimal and a common maximal element. 
This completes the determination of $R(P,Q_n)$ for posets $P$ with at most $4$ elements. 
If $P$ is an antichain $A_t$ on $t$ elements, we show that $R(A_t,Q_n)=n+3$ for $3\le t\le \log \log n$.
Additionally, we briefly survey proof techniques in the poset Ramsey setting $P$ versus $Q_n$.
\end{abstract}

\section{Introduction}
\subsection{Basic setting and background}
A \textit{partially ordered set}, or \textit{poset} for short, is a pair $(P,\le_P)$ of a set $P$ and a partial order $\le_P$ on this set, i.e.\ a binary relation that is transitive, reflexive and anti-symmetric. Usually we refer to a poset $(P,\le_P)$ just as $P$. The elements of $P$ are often called \textit{vertices}. 
If two vertices $A$ and $B$ are \textit{incomparable}, i.e.\ if  $A\not\le B$ and $A\not\ge B$, we write $A\parallel B$.
A poset $(P_2,\le_{P_2})$ is an \textit{(induced) subposet} of a poset $(P_1,\le_{P_1})$ if $P_2\subseteq P_1$ and for every two $X,Y\in P_2$, $X\le_{P_2} Y$ if and only if $X\le_{P_1} Y$.
If such a $P_2$ is isomorphic to some poset $P'$, we say that $P_2$ is a \textit{copy} of $P'$ in $P_1$.
Equivalently, $P_2$ is a copy of $P'$ in $P_1$ if $P_2$ is the image of an \textit{embedding} $\phi\colon P'\to P_1$, i.e.\ an injective function such that for every two $X,Y\in P'$, $X \leq_{P'}  Y$ if and only if $\phi(X)\leq_{P_1} \phi(Y)$.

The \textit{Boolean lattice} $Q_n$ is the poset whose vertices are the subsets of an $n$-element \textit{ground set} ordered by inclusion.
In this paper we consider colorings of the vertices of posets. A \textit{blue/red coloring} of a poset $P$, is a mapping $c\colon P\to \{\text{blue},\text{red}\}$. 
We say that a poset is \textit{monochromatic} if all its vertices have the same color. If all vertices are blue, we say that the poset is \textit{blue}. Similarly, if all vertices are red, the poset is \textit{red}.
\\

Ramsey-type problems are widely studied for graphs and hypergraphs and were extended to posets by Ne\v{s}et\v{r}il and R\"odl \cite{NR} in a general form. 
A special case of their studies asks to find for a fixed poset $P$, a hosting poset $W$ such that every blue/red-coloring of the elements of $W$ contains a monochromatic copy of $P$. 
Kierstead and Trotter \cite{KT} considered this setting with the goal of minimizing $p(W)$ for all posets $P$ with fixed $p(P)$, where $p$ is a poset parameter such as size, height or width.
Axenovich and Walzer \cite{AW} introduced a closely related Ramsey setting which recently attracted the attention of various researchers. 
For fixed posets $P_1$ and $P_2$ the \textit{poset Ramsey number} of $P_1$ versus $P_2$ is 
\begin{multline*}
R(P_1,P_2)=\min\{N\in\N \colon \text{ every blue/red coloring of $Q_N$ contains either}\\ 
\text{a blue copy of $P_1$ or a red copy of $P_2$}\}.
\end{multline*}
For the diagonal setting $P_1=P_2=Q_n$, the bounds on $R(Q_n,Q_n)$ were gradually improved to $2n+1\le R(Q_n,Q_n)\le n^2-n+2$, see chronologically Axenovich and Walzer~\cite{AW}, Cox and Stolee \cite{CS}, Lu and Thompson \cite{LT}, and Bohman and Peng \cite{BP}. The asymptotic behaviour in terms of $n$ remains an open problem. 
Note that for any $P_1$ and $P_2$ the poset Ramsey number is well-defined:
It is easy to see that every two posets $P_1$ and $P_2$ are induced subposets of $Q_n$ for large $n$, thus an upper bound on $R(Q_n,Q_n)$ implies the existence of $R(P_1,P_2)$.
For further results on diagonal poset Ramsey numbers $R(P,P)$ see e.g.\ Chen, Chen, Cheng, Li, and Liu \cite{CCCLL} and Walzer \cite{Walzer}. 

Another actively investigated setting of poset Ramsey numbers considers $R(Q_m,Q_n)$ for $m$ fixed and $n$ large. It is trivial to see that $R(Q_1,Q_n)=n+1$.
In the case $m=2$ it was shown that $R(Q_2,Q_n)=n+\Theta\big(\frac{n}{\log(n)}\big)$ with upper bound due to Gr\'osz, Methuku, and Tompkins \cite{GMT} 
and lower bound due to Axenovich and the author \cite{QnV}. 
For $m\ge 3$ only rough estimates are known, see Lu and Thompson \cite{LT}.
This open field of research as well as Erd\H{o}s-Hajnal-type questions on posets motivated a detailed study of the off-diagonal poset Ramsey number $R(P,Q_n)$ 
for fixed $P$ and large $n$ which is presented in a series of papers \cite{QnK}, \cite{QnN} including the present paper.

Other extremal problems on posets include rainbow Ramsey problems, see Chang, Gerbner, Li, Methuku, Nagy, Patk\'os, and Vizer \cite{CGLMNPV};
 Turan-type, most notable see Methuku and P\'alv\"olgyi \cite{MP}; 
 and saturation-type questions, which are discussed in a recent survey by Keszegh, Lemons, Martin, P\'alv\"olgyi, and Patk\'os \cite{KLMPP}. 

\subsection{Summary of results}\label{sec:old_results}

In order to understand the general framework of bounds on $R(P,Q_n)$ we consider two special posets.
The $V$-shaped poset $\cV$ has three distinct vertices $A, B,$ and $C$ with $C\le A$, $C\le B$, and $A\parallel B$. 
Its symmetric counterpart is the poset $\cLa$ which has three distinct vertices $A, B,$ and $C$ where $C\ge A$, $C\ge B$, and $A\parallel B$.
We say that a poset $P$ is \textit{non-trivial} if $P$ contains a copy of either $\cV$ or $\cLa$, otherwise we say that $P$ is \textit{trivial}.
It was shown by Axenovich and the author \cite{QnV} that two different asymptotic behaviours of $R(P,Q_n)$ emerge depending on whether the fixed $P$ is trivial or not.
Combined with a general lower bound of Walzer \cite{Walzer} and a general upper bound by Axenovich and Walzer \cite{AW} the following is known: 
\begin{theorem}[\cite{AW}\cite{QnV}\cite{Walzer}]\label{thm-MAIN}
Let $P$ be a trivial poset. Then for every $n$, $$n+h(P)-1\le R(P,Q_n) \le  n +h(P)+ \alpha(w(P))+1.$$
Let $P$ be a non-trivial poset. Then for sufficiently large $n$, $$n+\tfrac{1}{15}\tfrac{n}{\log n} \le R(P,Q_n) \le h(P)n + \dim_2(P).$$
\end{theorem}
Throughout this paper `$\log$' refers to the logarithm base $2$. 
The height $h(P)$ is the size of the largest set of pairwise comparable vertices in $P$, 
while the width $w(P)$ denotes the size of the largest set of pairwise incomparable vertices in $P$.
The \textit{$2$-dimension} $\dim_2(P)$ is the smallest dimension $N$ of a Boolean lattice $Q_N$ which contains a copy of $P$. It is a basic observation that this poset parameter is well-defined.
The \textit{Sperner number} $\alpha(t)$ is the minimal integer $N$ such that $\binom{N}{\lfloor N/2\rfloor}\ge t$.
It was determined almost exactly in an explicit form by Habib, Nourine, Raynaud and Thierry \cite{HNRT} who showed that
\begin{equation}
\alpha(t)\in\big\{\big\lfloor\log t+\tfrac{\log\log t}{2}\big\rfloor+1,\big\lfloor\log t+\tfrac{\log\log t}{2}\big\rfloor+2\big\}.\tag{$\ast$}
\end{equation}
We remark that posets $P$ of height $h(P)=1$, i.e.\ \textit{antichains}, are trivial, however for such $P$ the general upper bound $R(P,Q_n) \le h(P)n + \dim_2(P)$ by Axenovich and Walzer \cite{AW} is stronger than the bound $R(P,Q_n) \le  n +h(P)+ \alpha(w(P))+1$ given in Theorem \ref{thm-MAIN}.
\\

Another important contribution of Walzer \cite{Walzer} in the analysis of $R(P,Q_n)$ is providing a bound for \textit{parallel compositions} $P$. 
Given a poset $\QQ$, two subposets $P_1,P_2\subseteq \QQ$ are \textit{parallel} if they are element-wise incomparable.
We denote by $P_1+P_2$ the \textit{parallel composition} of two posets $P_1$ and $P_2$, that is the poset consisting of a copy of $P_1$ and a copy of $P_2$ which are disjoint and parallel. 
In the literature the parallel composition is also referred to as \textit{independent union}. Note that this operation is commutative and associative.

\begin{theorem}[\cite{Walzer}]\label{thm_union}  
Let $\ell\ge 2$ and let $P_1,P_2,\dots,P_\ell$, and $Q$ be arbitrary posets.
Let $\PP=P_1+P_2+\dots+P_\ell$ be the parallel composition of $P_1,\dots,P_\ell$. Then
$$R(\PP,Q)\le  \max_{j\in[\ell]} \big\{R(P_j,Q) \big\} +\alpha(\ell) \le \max_{j\in[\ell]} \big\{R(P_j,Q) \big\} +\log (\ell)+\tfrac12 \log\log(\ell)+2.$$
\end{theorem}

Collecting known results and providing several new bounds in this paper, 
we obtain bounds on $R(P,Q_n)$ which are asymptotically tight in the two leading additive terms for all posets $P$ on at most $4$ vertices (of which there are 19 up to symmetry).
Moreover, we exactly determine $R(P,Q_n)$ for all trivial $P$ on at most $4$ vertices. An overview of the bounds is given in Table \ref{table}.
In every row of the table a poset $P$ is defined by its Hasse diagram and labelled using the notation of this paper. Formal definitions of all posets are stated in Sections \ref{sec:new} and \ref{sec:notation}. 
Some posets have alternative names used in the literature, these are additionally mentioned in the table.

\renewcommand{\arraystretch}{1.1}
\def\figscale{0.41}
\def\rowhgt{20pt}
\def\rowhgtb{19pt}

\begin{table}\label{table}
\begin{center}
\begin{tabular}{| l | r  l | l | l |}
    \hline \rule{0pt}{13pt}
    &\multicolumn{2}{c|}{\textbf{poset $P$}} & $R(P,Q_n)$ & \textbf{proof} \\ \hline\hline \rule{0pt}{\rowhgt}
    \includegraphics[scale=\figscale]{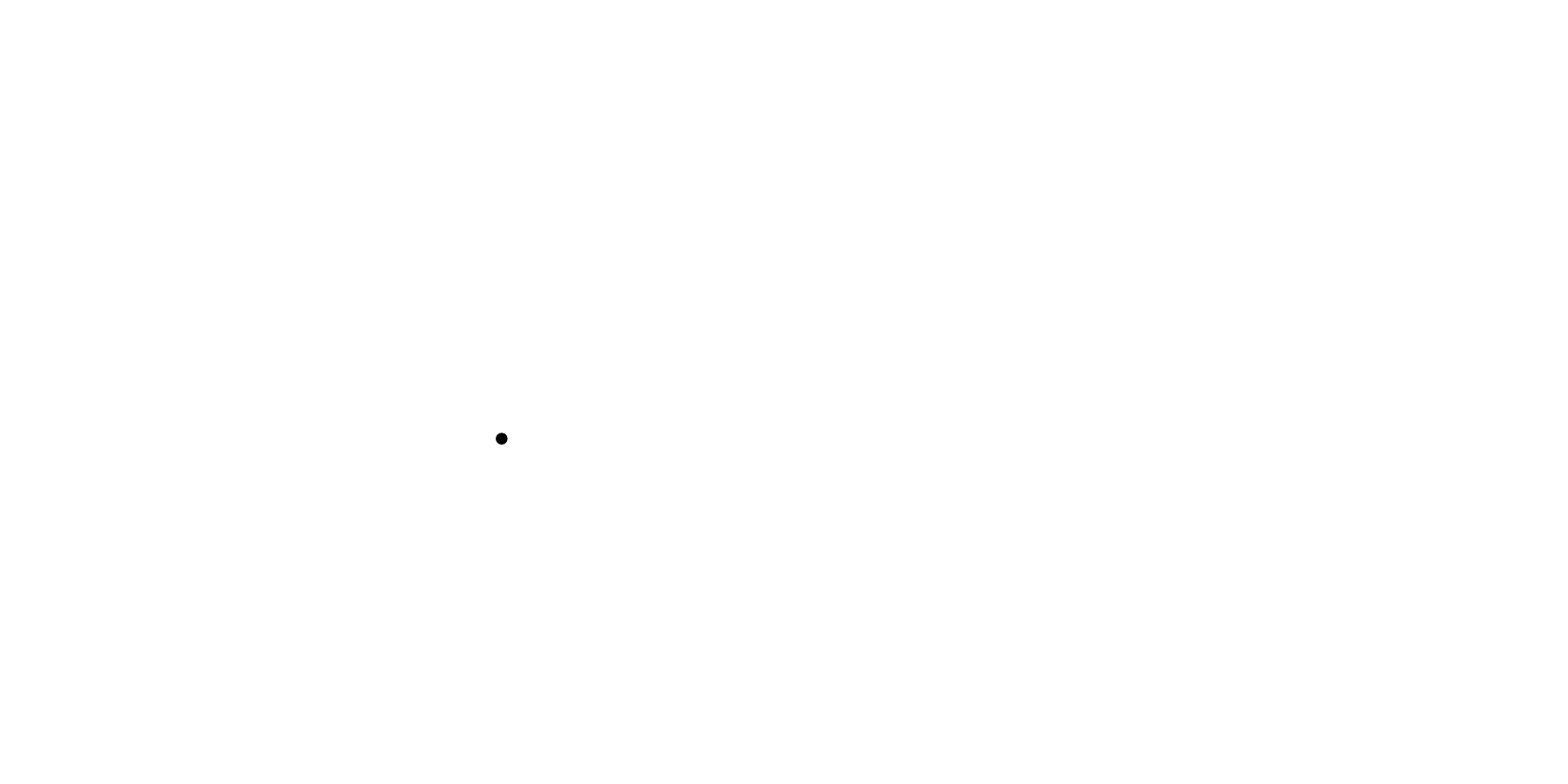} & 
    $C_1$& $=Q_0$ & $n+0$ & trivial\\ \hline \hline \rule{0pt}{\rowhgt}
    
    \includegraphics[scale=\figscale]{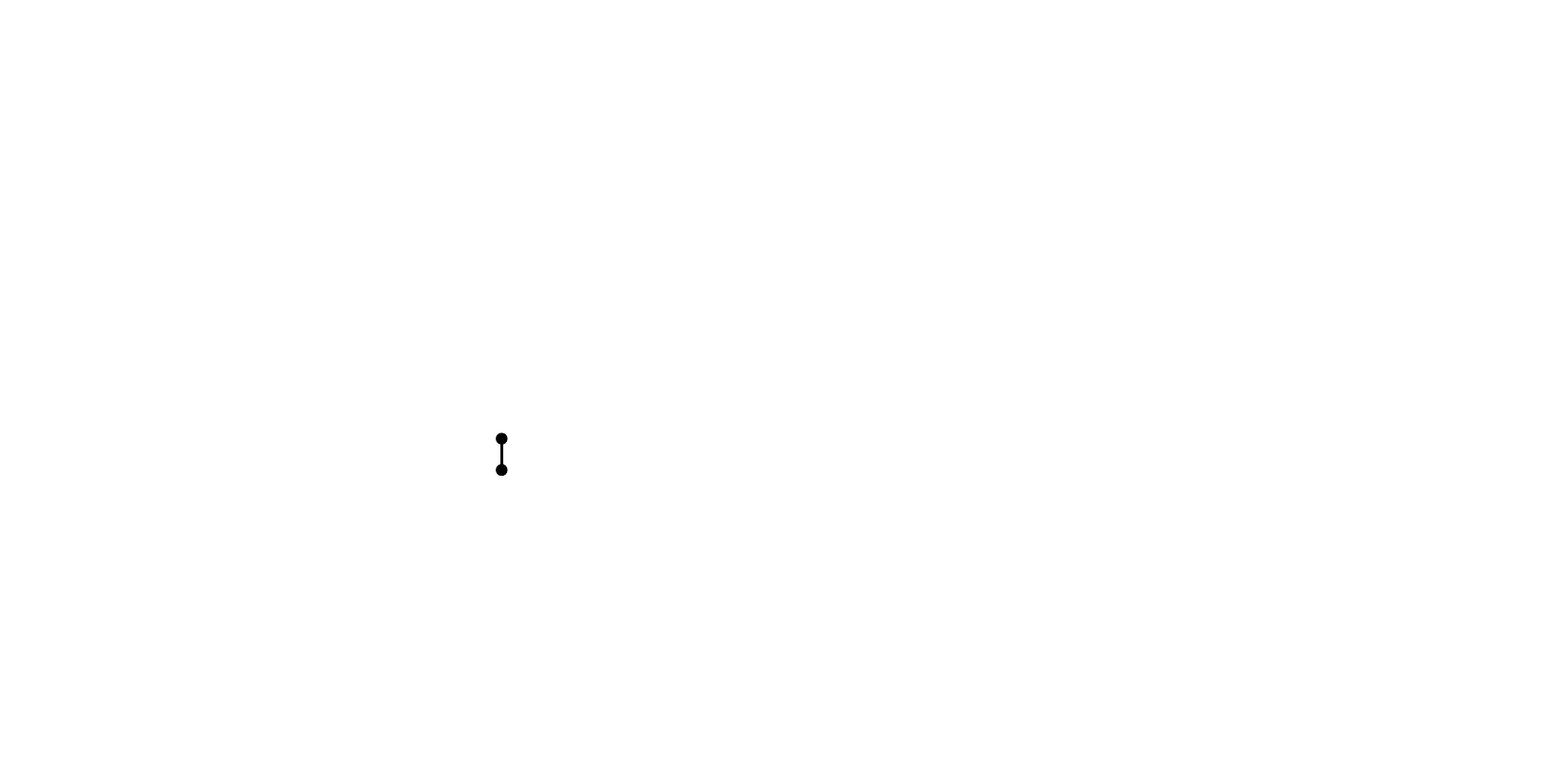} & 
    $C_2$& $=Q_1$ & $n+1$ & Thm.~\ref{thm_cP_UB} \\ \hline \rule{0pt}{\rowhgt}
    
    \includegraphics[scale=\figscale]{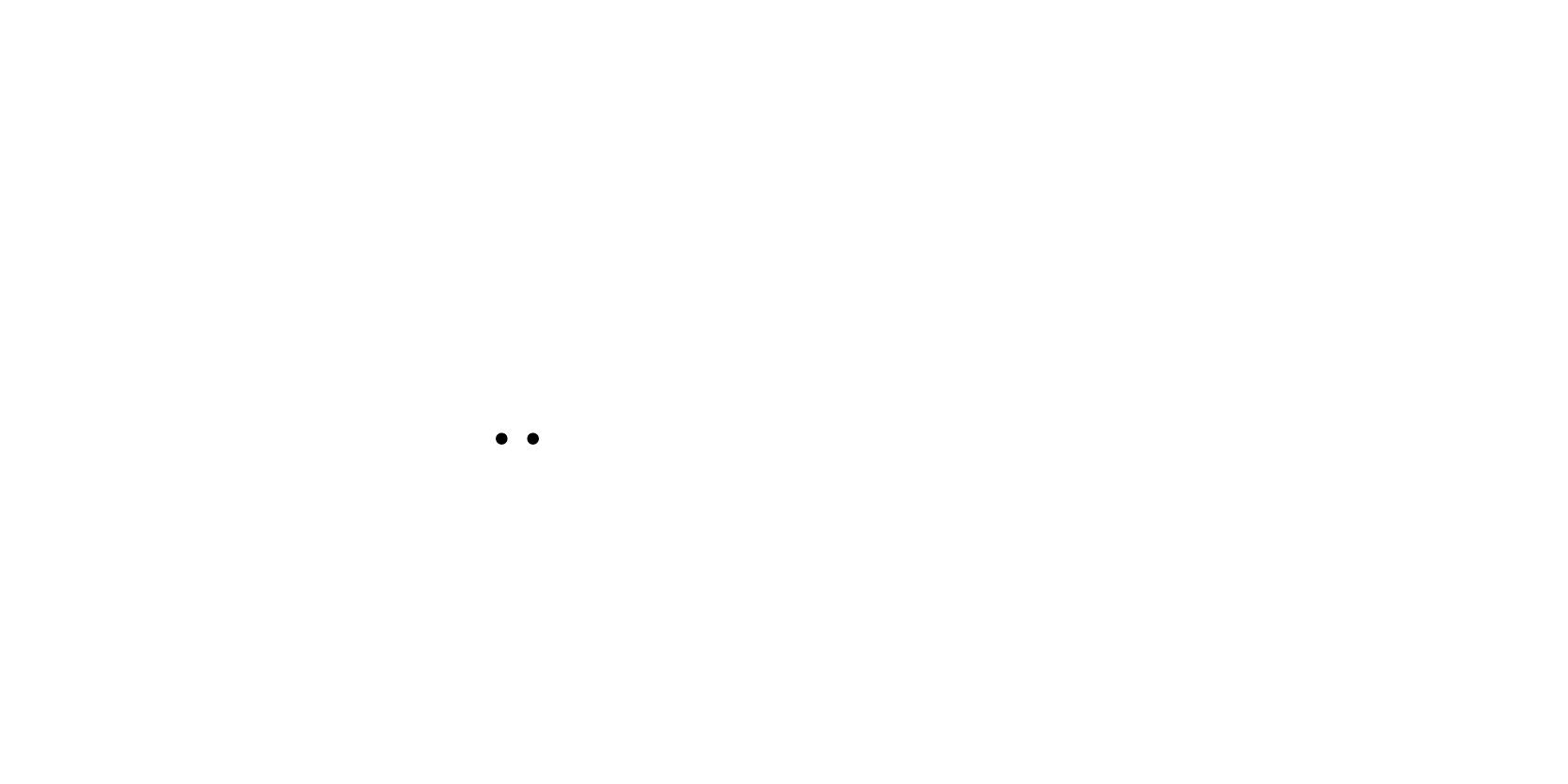} & 
    $A_2$&$=2C_1$ & $n+2$ &Thm.~\ref{thm_cP_UB} \\ \hline\hline \rule{0pt}{\rowhgt}
    
    \includegraphics[scale=\figscale]{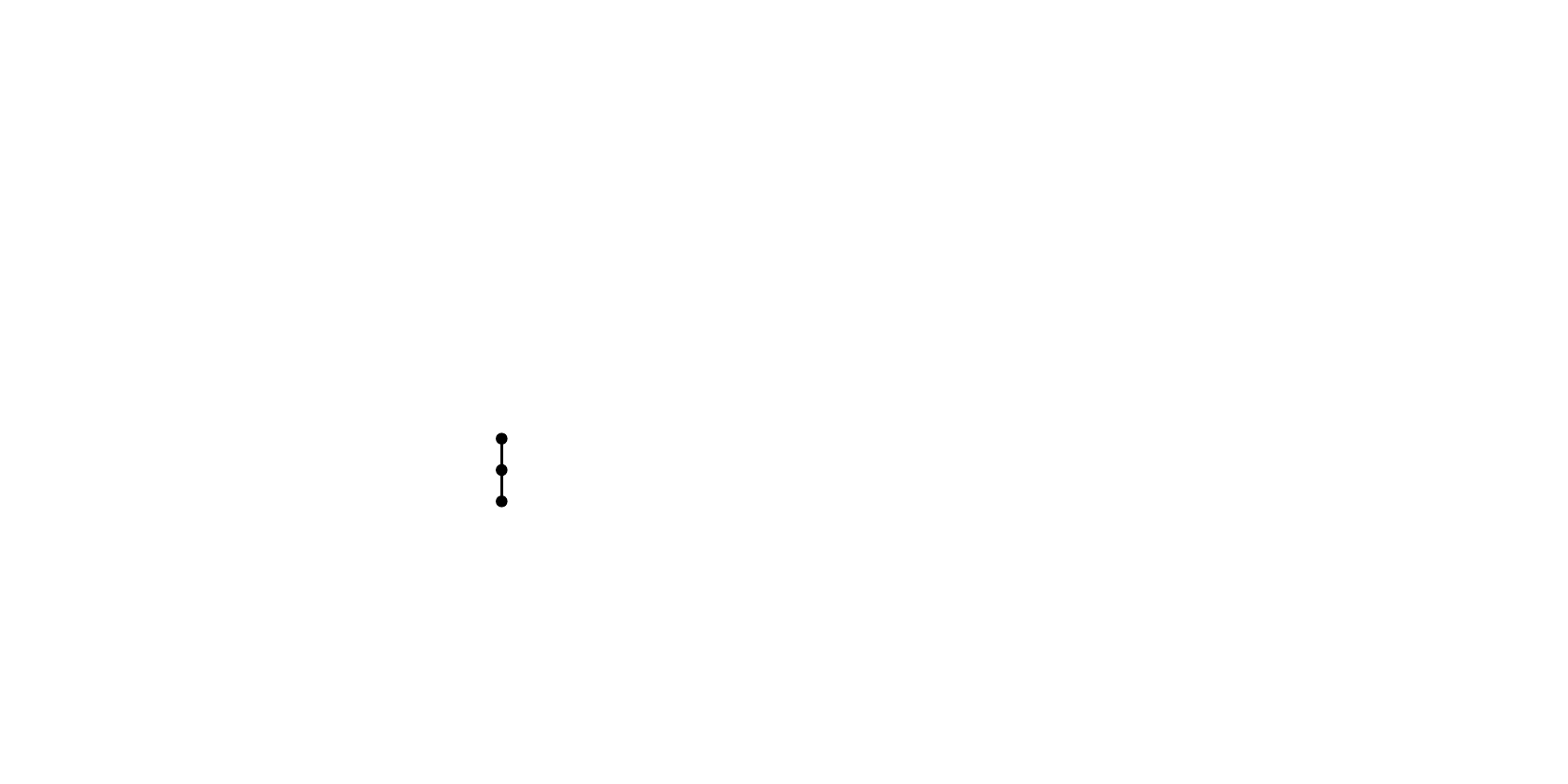} & 
    $C_3$& & $n+2$ & Thm.~\ref{thm_cP_UB}  \\ \hline \rule{0pt}{\rowhgt}
    
    \includegraphics[scale=\figscale]{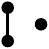} & 
    $\CC_{2,1}$ & & $n+3$ & Thm.~\ref{thm_cP_UB}  \\ \hline \rule{0pt}{\rowhgt}
    
    \includegraphics[scale=\figscale]{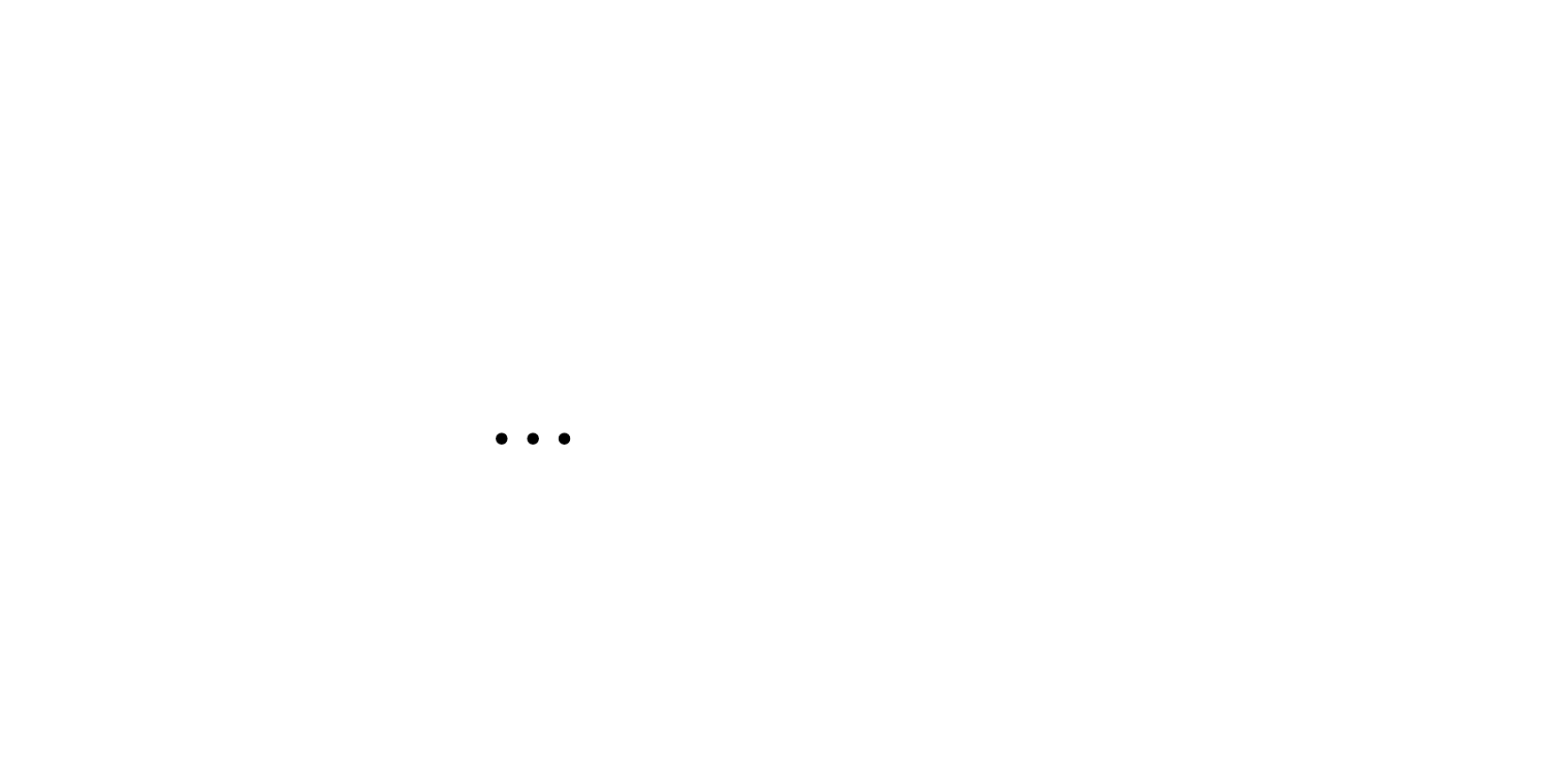} & 
    $A_3$&$=3C_1$ & $n+3$ & Thm.~\ref{thm:antichain} \\ \hline \rule{0pt}{\rowhgtb}
    
    \includegraphics[scale=\figscale]{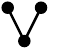} & 
    $\cV$&$=K_{1,2}$	 &$n+\frac{c(n)n}{\log(n)}$, \ $\frac{1}{15}\le c(n)\le 1+o(1)$ & \cite{QnV}\\ \hline\hline \rule{0pt}{\rowhgt}
    
    \includegraphics[scale=\figscale]{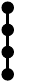} & 
    $C_4$& & $n+3$ & Thm.~\ref{thm_cP_UB} \\ \hline \rule{0pt}{\rowhgt}
    
    \includegraphics[scale=\figscale]{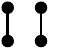} & 
    $\CC_{2,2}$&$=2C_2$ & $n+3$ & Thm.~\ref{thm_cP_UB} \\ \hline    \rule{0pt}{\rowhgt}
    
    \includegraphics[scale=\figscale]{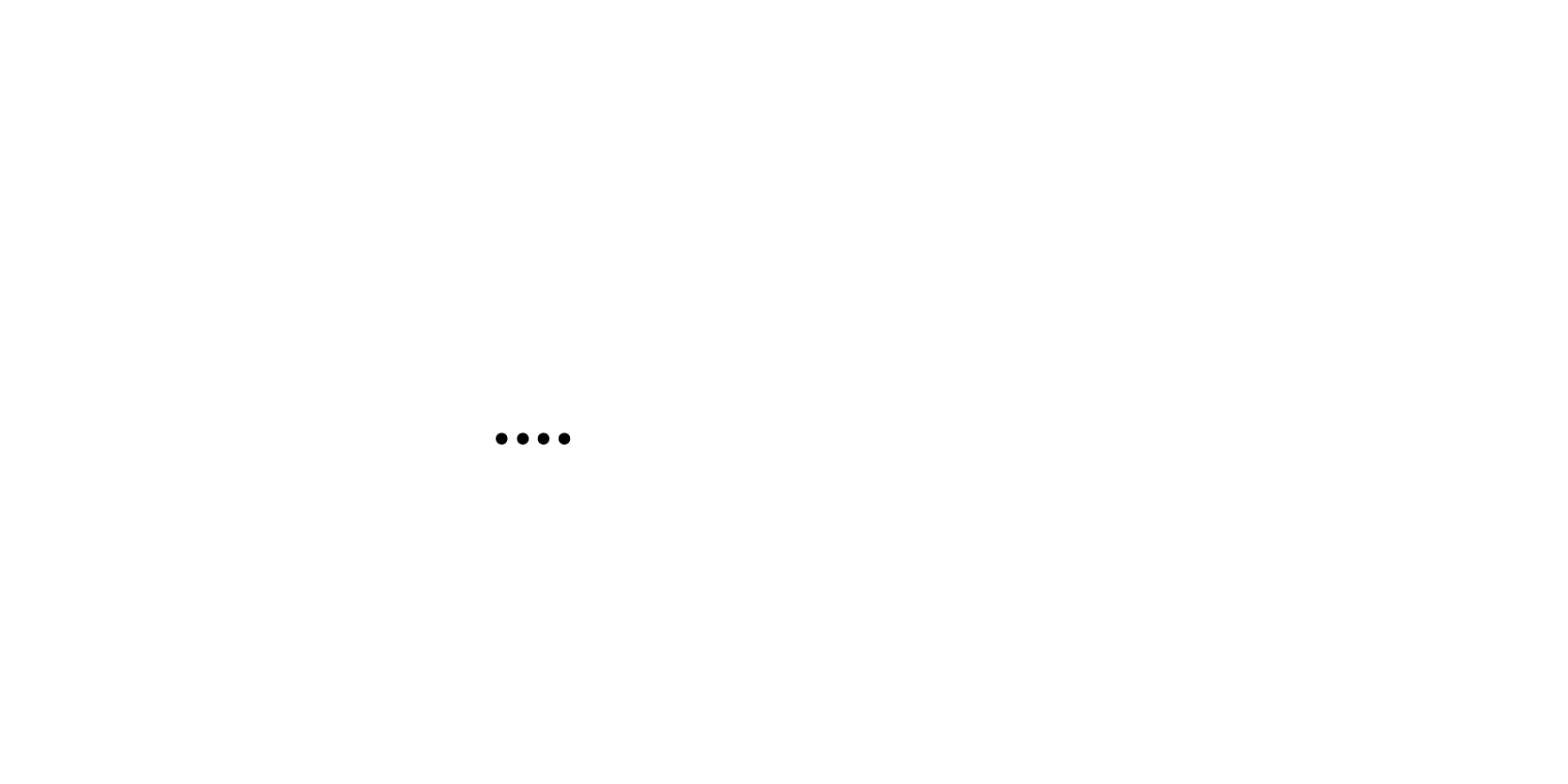} & 
    $A_4$&$=4C_1$ & $n+3$ & Thm.~\ref{thm:antichain} \\ \hline    \rule{0pt}{\rowhgt}
    
    \includegraphics[scale=\figscale]{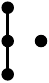} & 
    $\CC_{3,1}$ & & $n+4$ & Thm.~\ref{thm_cP_UB}  \\ \hline \rule{0pt}{\rowhgt}
    
    \includegraphics[scale=\figscale]{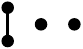} & 
    $\CC_{2,1,1}$ & & $n+4$ & Thm.~\ref{lem_CA} \\ \hline \rule{0pt}{\rowhgtb}
    
    \includegraphics[scale=\figscale]{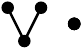} & 
    $\cV+C_1$& & $n+\frac{c(n)n}{\log(n)}$, \ $\frac{1}{15}\le c(n)\le 1+o(1)$ &
    \begin{tabular}[c]{@{}l@{}}LB: Thm.~\ref{thm-MAIN} (\cite{QnV}),\\UB: \cite{QnV}, Thm.~\ref{thm_union} (\cite{Walzer})\end{tabular}\\ \hline \rule{0pt}{\rowhgtb}
      
    \includegraphics[scale=\figscale]{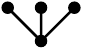} & 
    $K_{1,3}$&$=V_3$ & $n+\frac{c(n)n}{\log(n)}$, \ $\frac{1}{15}\le c(n)\le 1+o(1)$ &
     \begin{tabular}[c]{@{}l@{}}LB: Thm.~\ref{thm-MAIN} (\cite{QnV}),\\UB: \cite{QnK}\end{tabular}\\ \hline \rule{0pt}{\rowhgtb}
    
    \includegraphics[scale=\figscale]{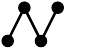} & 
    $\cN$& & $n+\frac{c(n)n}{\log(n)}$, \ $\frac{1}{15}\le c(n)\le 1+o(1)$ &
     \begin{tabular}[c]{@{}l@{}}LB: Thm.~\ref{thm-MAIN} (\cite{QnV}),\\UB: \cite{QnN}\end{tabular}\\ \hline \rule{0pt}{\rowhgtb}
    
    \begin{tabular}[c]{@{}l@{}}\includegraphics[scale=\figscale]{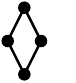}\end{tabular} & 
    $Q_2$&$=K_{1,2,1}$ & $n+\frac{c(n)n}{\log(n)}$, \ $\frac{1}{15}\le c(n)\le 2+o(1)$ & 
    \begin{tabular}[c]{@{}l@{}}LB: Thm.~\ref{thm-MAIN} (\cite{QnV}),\\UB: \cite{GMT}\end{tabular}\\ \hline \rule{0pt}{\rowhgtb}
    
    
    \begin{tabular}[c]{@{}l@{}}\includegraphics[scale=\figscale]{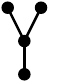}\end{tabular}& 
    $K_{1,1,2}$&$=Y$ & $n+\frac{c(n)n}{\log(n)}$, \ $\frac{1}{15}\le c(n)\le 2+o(1)$ &
     \begin{tabular}[c]{@{}l@{}}LB: Thm.~\ref{thm-MAIN} (\cite{QnV}),\\UB: \cite{QnK}\end{tabular}\\ \hline \rule{0pt}{\rowhgtb}
    
    \begin{tabular}[c]{@{}l@{}}\includegraphics[scale=\figscale]{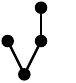}\end{tabular}& 
    $\cJ$&$=Q_2^-$ & $n+\frac{c(n)n}{\log(n)}$, \ $\frac{1}{15}\le c(n)\le 2+o(1)$ &
     \begin{tabular}[c]{@{}l@{}}LB: Thm.~\ref{thm-MAIN} (\cite{QnV}),\\UB: Cor.~\ref{cor:J}\end{tabular}\\ \hline \rule{0pt}{\rowhgtb}

    \includegraphics[scale=\figscale]{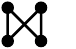} & 
    $K_{2,2}$&$=\Bowtie$ & $n+\frac{c(n)n}{\log(n)}$, \ $\frac{1}{15}\le c(n)\le 4+o(1)$ &
     \begin{tabular}[c]{@{}l@{}}LB: Thm.~\ref{thm-MAIN} (\cite{QnV}),\\UB: \cite{QnK}\end{tabular}\\ \hline 
  \end{tabular}
  \caption{Off-diagonal poset Ramsey bounds for small $P$ and reference to the proofs of lower bound (LB) and upper bound (UB)}
\label{table}
\end{center} 
\end{table}

\subsection{New results}\label{sec:new}

A \textit{chain} $C_t$ of length $t$ is a poset on $t$ vertices forming a linear order. 
For $s,t\in\N$, let $\cS\cD_{s,t}$ denote the \textit{$(s,t)$-subdivided diamond}, the poset obtained from two disjoint and element-wise incomparable chains of length $s$ and $t$, respectively, by adding a common minimal vertex and a common maximal vertex, i.e.\ a vertex which is smaller than all others and a vertex which is larger than all other, see Figure \ref{fig:SD}. Note that $\cS\cD_{1,1}=Q_2$. 
Our first result shows that for subdivided diamonds the lower bound obtained from Theorem \ref{thm-MAIN} is asymptotically tight in a strong sense.

\begin{figure}[h]
\centering
\includegraphics[scale=0.5]{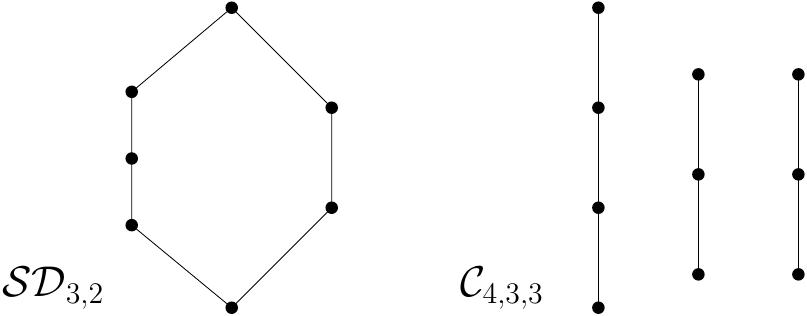}
\caption{Hasse diagrams of $\cS\cD_{3,2}$ and $\CC_{4,3,3}$}
\label{fig:SD}
\end{figure}

\begin{theorem}\label{thm:subdiv_diamond}\label{thm:SD}
Let $s$ and $t$ be fixed natural numbers. Then for sufficiently large $n$,
$$R(\cS\cD_{s,t},Q_n)\le n + \frac{\big(2+o(1)\big)n}{\log n}.$$
\end{theorem}

\noindent Note that the poset $\cJ$, as defined in the penultimate line of Table \ref{table}, is an induced subposet of $\cS\cD_{1,2}$. 
Thus, Theorem \ref{thm:SD} implies:

\begin{corollary}\label{cor:J}
For $n$ sufficiently large, $R(\cJ,Q_n)\le R(\cS\cD_{1,2})\le n+ \big(2+o(1)\big)\frac{n}{\log n}$.
\end{corollary}

\noindent We remark that Theorem \ref{thm:SD} can be generalized, but the argument is rather technical, so we omit it here. 
Generalizing subdivided diamonds we can show that for every poset $P$ with width $w(P)=2$ and which contains no copy of the N-shaped $4$-element poset $\cN$,
$$R(P,Q_n)= n + O\left(\frac{n}{\log n}\right).$$
This bound follows from Theorem \ref{thm:SD} of the present article, Theorem 1 and Corollary 6 of \cite{QnK} as well as a characterization by Valdes \cite{Valdes} stating that a poset is $\cN$-free if and only if it is series-parallel.
\\


It is a simple observation that a poset is trivial if and only if it is a parallel composition of chains $C_{t_1}, \dots, C_{t_\ell}$.
We say that this is the \textit{chain composition} with parameters $t_1,\dots,t_\ell$,
$$\CC_{t_1,t_2,\dots,t_\ell}=C_{t_1} + C_{t_2} + \dots + C_{t_\ell}.$$
Throughout the paper we use the convention that $t_1\ge t_2 \ge \dots \ge t_\ell$. 
Theorem \ref{thm-MAIN} provides that $n+t_1-1\le R(\CC_{t_1,t_2,\dots,t_\ell},Q_n) \le n +t_1+ \alpha(\ell)+1.$ Here we show the following.

\begin{theorem}\label{thm_cP_UB}\label{thm:CC}
Let $n,t_1,t_2\in\N$ such that $t_1\ge t_2$. Then
$$R(C_{t_1},Q_n)=n+t_1-1\qquad \text{and}\qquad R(\CC_{t_1,t_2},Q_n)=n+t_1+1.$$
\end{theorem}

\begin{theorem}\label{lem_CA}\label{thm:CCC}
Let $n,t,t'\in\N$ with $t-1\ge t'$. Then
$$R(\CC_{t,t-1,t'},Q_n)= n+t+2.$$
\end{theorem}

\noindent Theorem \ref{thm_union} implies an improvement of the general lower bound for trivial posets.

\begin{corollary}\label{cor:LB}
Let $n\in \N$. Let $P$ be a trivial poset. If $w(P)=1$, then $R(P,Q_n)=n+h(P)-1$. If $w(P)\ge 2$, then 
$R(P,Q_n) \ge n+h(P)+1$.
\end{corollary}

\noindent Note that every poset with $w(P)=1$ is trivial. By Theorem \ref{thm-MAIN}, the bound for $w(P)\ge2$ also holds for non-trivial posets if $n$ is large. 
However, for small $n$, Corollary \ref{cor:LB} does not extend to non-trivial posets, for example it can be easily checked that $R(\cV,Q_1)=3<4$.
\\

An \textit{antichain} $A_t$ is the parallel composition $C_{1,1,\dots,1}$ of $t$ single vertices, i.e.\ the poset consisting of $t$ pairwise incomparable vertices. 
Since $h(A_t)=1$ and $\dim_2(A_t)=\alpha(t)$ by the famous Sperner's theorem \cite{S}, the bounds in \cite{AW} and $(\ast)$ imply that 
$$n\le R(P,Q_n)\le n+\alpha(t)\le n+\log t+\tfrac{\log\log t}{2}+2.$$
In this paper we exactly determine $R(A_t,Q_n)$, not only for fixed $t$ but also if $t$ grows at most double-logarithmic in terms of $n$.

\begin{theorem}\label{thm_antichain}\label{thm:antichain}
For every two integers $t$ and $n$ with $3\le t\le \log\log n$, $$R(A_t,Q_n)= n+3.$$
\end{theorem}
\noindent In fact, our result holds for $n\ge 2^{2^{t-2}}-2$, which is a slightly weaker precondition than $t\le \log\log n$.
Besides that, we prove that if $t$ is large in terms of $n$, the poset Ramsey number $R(A_t,Q_n)$ exceeds $n+3$.

\begin{theorem}\label{thm_antichain2}
Let $n,r,t\in\N$ such that $t> \binom{n+2r+1}{r}$. Then $R(A_t,Q_n)\ge n+2r+2$. \linebreak
In particular, if $t\ge n+4$, then $R(A_t,Q_n)\ge n+4$.
\end{theorem}

\noindent Stated explicitly in terms of $n$ and $t$, Theorems \ref{thm-MAIN} and \ref{thm_antichain2}, and $(\ast)$ provide the following.

\begin{corollary}\label{cor_ac_asym}
For $n,t\in\N$ with $n\ge 3$ and $t\ge 2$,
$$n+\frac{2\log t}{3+\log n}\le R(A_t,Q_n)\le n +\log t+\frac{\log\log t}{2}+2.$$
\end{corollary}

Our paper is structured as follows.
In Section \ref{sec:gen} we introduce basic definitions and notation, and discuss preliminary observations.
Section \ref{sec:non-triv} gives a proof of Theorem~\ref{thm:SD}.
In Section~\ref{sec:triv}, we present proofs of Theorems \ref{thm:CC} and \ref{thm:CCC}. Section \ref{sec:antichain} focusses on antichains, and gives a proof of Theorems \ref{thm_antichain} and \ref{thm_antichain2} as well as Corollary \ref{cor_ac_asym}.
Finally, in Section \ref{sec:survey} we summarize known proof techniques for bounding off-diagonal poset Ramsey numbers and collect some open problems.

In this paper we omit floors and ceilings where appropriate. The set of the first $n$ natural numbers is denoted by $[n]=\{1,\dots,n\}$.

\section{Preliminaries} \label{sec:gen}

\subsection{Poset notation and classic results}\label{sec:notation}

In the previous section we stated formal definitions of the \textit{Boolean lattice} $Q_n$, the $V$-shaped poset $\cV$, the \textit{chain} $C_t$, the antichain $A_t$, the \textit{$(s,t)$-subdivided diamond} $\cS\cD_{s,t}$ and the \textit{chain composition} $\CC_{t_1,t_2,\dots,t_\ell}$. Besides those we use the following notation for posets. Examples of all mentioned posets are given in Table \ref{table}.

The poset $\cN$ consists of four distinct vertices $A,B,C$, and $D$ such that $A\le C$, $C\ge B$, $B\le D$, $A||B$, $A||D$, and $C||D$.
The hook-shaped poset $\cJ$ has distinct vertices $A,B,C$, and $D$ where $B\le C\le D$, $A\ge B$, $A||C$, and $A||D$.

The \textit{complete $\ell$-partite poset} $K_{t_1,\dots,t_\ell}$ is a poset on $\sum_{i=1}^\ell t_i$ many vertices defined as follows. 
For each index $i\in[\ell]$, there is a set of $t_i$ distinct vertices, called \textit{layer} $i$.
Every pair of vertices in the same layer is incomparable. For any two vertices $X$ and $Y$ belonging to different layers $i_X$ and $i_Y$ with $i_X<i_Y$, we have $X\le Y$.
\\

In this paper, we commonly consider a Boolean lattice $Q_n$ with a specified ground set.
Given a set $\cX$, the \textit{Boolean lattice} $\QQ(\cX)$ is the poset on all subsets of $\cX$ equipped with set inclusion relation. The \textit{dimension} of $\QQ(\cX)$ is $|\cX|$.
Let $\QQ$ be a Boolean lattice of dimension~$N$. For $\ell\in\{0,\dots,N\}$, we say that \textit{layer $\ell$} of $\QQ$ is the set of elements $\{Z\in\QQ : ~ |Z|=\ell\}$.
Note that $\QQ$ consists of $N+1$ layers and that each layer induces an antichain in $\QQ$.
A blue/red coloring of a Boolean lattice is \textit{layered} if within each layer every vertex receives the same color.
A vertex $X$ is the \textit{minimum} of a poset $P$ if $X\le Z$ for every $Z\in P$. Similarly, a \textit{maximum} is a vertex $X\in P$ such that $X\ge Z$ for every $Z\in P$.
\\



The following classic result is known as Dilworth's theorem \cite{Dil}.
\begin{theorem}[Dilworth \cite{Dil}]\label{dilworths}
If a poset $P$ contains no antichain of size $t$, then all vertices of $P$ can be covered by $t-1$ chains.
\end{theorem}

A chain in an $N$-dimensional Boolean lattice is said to be \textit{symmetric} 
if it consists of vertices $X_{\ell}\subset \dots \subset X_{N-\ell}$ for some $\ell\in\N$ such that $|X_i|=i$ for all $i\in\{\ell,\dots,N-\ell\}$.
De Bruijn, Tengbergen and Kruyswijk \cite{BEK} showed the following decomposition result.

\begin{theorem}[\cite{BEK}]\label{thm_scd}
The vertices of an $N$-dimensional Boolean lattice can be decomposed into pairwise disjoint, symmetric chains.
\end{theorem}

\subsection{Red Boolean lattice versus blue chain}

In this paper we often consider the Boolean lattice $\QQ(\cX\cup\cY)$ where $\cX$ and $\cY$ are two disjoints sets with $\cY\neq\varnothing$.
We denote a linear ordering $\tau$ of $\cY$ where $y_1<_\tau y_2<_\tau \dots <_\tau y_k$ by a sequence $\tau=(y_1,\dots,y_k)$ implying that $\cY=\{y_1,\dots,y_k\}$.
Given a linear ordering $\tau=(y_1,\dots,y_k)$ of $\cY$, a \textit{$\cY$-chain} corresponding to $\tau$ is a $(k+1)$-element chain $C$ in $\QQ(\cX\cup\cY)$ on vertices
$$C=\big\{X_0\cup \varnothing,X_1\cup\{y_1\},X_2\cup\{y_1,y_2\},\dots,X_{k}\cup \cY\big\},$$
 where $X_0 \subseteq X_1 \subseteq\dots \subseteq X_k\subseteq \cX$. 
Note that $\cY$-chains corresponding to distinct linear orderings of $\cY$ are distinct.
The following \textit{Chain Lemma} was proved implicitly by Gr\'osz, Methuku and Tompkins \cite{GMT}. 
Since it is a fundamental lemma in determining $R(P,Q_n)$, we restate its self-contained proof by following the lines of Axenovich and the author, see Lemma 8 in~\cite{QnV}.

\begin{lemma}[\cite{GMT}, Chain Lemma]\label{lem:chain}\label{chain_lem}
Let $\cX$ and $\cY$ be disjoint sets with $|\cX|=n$ and $|\cY|=k$. Let $\QQ(\cX\cup\cY)$ be a blue/red colored Boolean lattice.
Fix a linear ordering $\tau=(y_1,\dots,y_k)$  of $\cY$. Then there exists in $\QQ(\cX\cup\cY)$ either a red copy of $Q_n$, or a blue $\cY$-chain corresponding to $\tau$.
\end{lemma}

\begin{proof}
We denote the first $i$ elements of $\cY$ with respect to $\tau$ by $\cY[0]=\varnothing$ and $\cY[i]=\{y_1,\dots,y_i\}$ for $i\in[k]$.
Assume that there does not exist a blue $\cY$-chain corresponding to $\tau$. We show that there is a red copy of $Q_n$.
Recall that a copy of $P$ in $Q$ is the image of an embedding $\phi\colon P \to Q$.
In the following for every $X\subseteq\cX$, we recursively define a label $\ell_X\in\{0,\dots,k\}$ such that the function $$\phi\colon \QQ(\cX)\to \QQ(\cX\cup\cY), \ \phi(X)=X\cup \cY[\ell_X]$$
is an embedding with monochromatic red image. The image of such an embedding is a red copy of $Q_n$ as required. 
We choose labels $\ell_X$, $X\subseteq \cX$, with the following properties: 
\begin{enumerate}[label=(\roman*)]
\item For any $U\subseteq X$, \ $\ell_{U}\le \ell_{X}$.
\item There is a blue chain $C^X$ of length $\ell_X$ ``below'' the vertex $X\cup\cY[\ell_X]$, i.e.\ in the Boolean lattice $Q^{X}:=\QQ(X\cup  \cY[\ell_X])$ there is a blue $\cY[\ell_X]$-chain corresponding to the linear ordering $(y_1,\dots,y_{\ell_X})$.
\item The vertex $X\cup\cY[\ell_X]$ is colored red.
\end{enumerate}

First, consider the subset $\varnothing\subseteq\cX$. Let\ $\ell_\varnothing$ be the smallest index $\ell$, $0\le \ell\le k$, such that the vertex $\varnothing\cup\cY[\ell]$ is red.
If there is no such $\ell$, then\ the vertices $\varnothing\cup\cY[0], \dots, \varnothing\cup\cY[k]$ form a blue $\cY$-chain corresponding to $\tau$, a contradiction.
It is immediate that Properties (i) and (iii) hold for $\ell_\varnothing$. If $\ell_\varnothing=0$, then (ii) holds trivially. 
If $\ell_\varnothing\ge 1$, then vertices $\varnothing\cup\cY[0], \dots, \varnothing\cup\cY[\ell_{\varnothing}-1]$ form a blue chain of length $\ell_\varnothing$ as required for (ii).
\\

Now consider an arbitrary non-empty $X\subseteq\cX$ and suppose that for every $U\subsetneql X$ we already defined $\ell_{U}$ with Properties (i)-(iii). 
Let $\ell'_X=\max_{\{U\subsetneql X\}} \ell_U$ and fix some $W\subsetneql X$ with $\ell_W=\ell'_X$. Let $C^W$ be the blue chain obtained by Property (ii) for $W$.
We define $\ell_X$ as the smallest integer~$\ell$ with $\ell'_X\le \ell \le k$ such that the vertex $X\cup\cY[\ell]$ is red in the coloring of $\QQ(\cX\cup\cY)$.
If such an $\ell$ does not exist, i.e.\ if there is no red $X\cup\cY[\ell]$, the vertices $X\cup\cY[\ell'_X],\dots,X\cup\cY[k]$ form a blue chain of length $k-\ell'_X+1$. 
The chain $C^W$ is a blue chain ``below'' $W\cup\cY[\ell_W]$. Note that $W\cup\cY[\ell_W]\subsetneql X\cup\cY[\ell'_X]$, thus both chains combine to a chain $C^X$ of length $k+1$. 
It is easy to see that $C^X$ is a $\cY$-chain corresponding to $\tau$, so we arrive at a contradiction.
Thus, $\ell_X$ is well-defined. 

It is immediate that Property (iii) holds for $\ell_X$.
Furthermore, for $U\subsetneql X\subseteq \cX$ we have $\ell_{U}\le \ell'_X\le \ell_X$, thus (i) holds.
It remains to verify Property (ii) for $\ell_X$.
Recall that $W\subsetneql X$ such that $\ell_W=\ell'_X$. If $\ell_X=\ell'_X$, the chain $C^X:=C^W$ is as required. If $\ell_X\neq \ell'_X$, the chain $C^W$ together with vertices $X\cup\cY[\ell'_X], \dots, X\cup\cY[\ell_X-1]$ is a blue chain of length $\ell_X$, which verifies Property (ii).
\\

We use the labels $\ell_X$, $X\subseteq\cX$, to define an embedding of a Boolean lattice $Q_n$ in $\QQ(\cX\cup\cY)$.
Let $\phi\colon \QQ(\cX)\to \QQ(\cX\cup\cY)$ with $\phi(X)=X\cup \cY[\ell_X]$.
Property (iii) implies that the image of $\phi$ is colored monochromatically red.
We show that for any two $X_1,X_2\subseteq\cX$, it holds that $X_1\subseteq X_2$ if and only if $\phi(X_1)\subseteq\phi(X_2)$.
Indeed, if $\phi(X_1)\subseteq\phi(X_2)$, it is immediate that $X_1=\phi(X_1)\cap\cX \subseteq \phi(X_2)\cap\cX = X_2$.
Conversely, if $X_1\subseteq X_2$, then by Property (i) we see that $\ell_{X_1}\le \ell_{X_2}$. Thus $X_1\cup\cY[\ell_{X_1}]\subseteq X_2\cup\cY[\ell_{X_2}]$.
Therefore, $\phi$ is an embedding of $\QQ(\cX)$ with red image, so in particular $\QQ(\cX\cup\cY)$ contains a red copy of $Q_n$.
\end{proof}

The red copy of $Q_n$ obtained in this lemma has a strong additional property with respect to $\cX$ which is not needed here. 
For further details we refer the reader to Lemma 8 of \cite{QnV}. The following corollary is a simplified version of the Chain Lemma.

\begin{corollary}\label{cor:chain}
Let $n$ and $k$ be positive integers. Let $\QQ$ be a blue/red colored Boolean lattice of dimension $n+k$. 
Then $\QQ$ contains a red copy of $Q_n$ or a blue chain of length $k+1$.
\end{corollary}
\noindent Note that Corollary \ref{cor:chain} immediately implies the first upper bound in Theorem \ref{thm:CC}.

\subsection{Embedding of a Boolean lattice}

Axenovich and Walzer \cite{AW} showed that an embedding of a small Boolean lattice into a larger Boolean lattice has the following nice property, see Theorem 8 of \cite{AW}.

\begin{lemma}[\cite{AW}]\label{embed_lem} 
Let $n\in\N$ and let $\cZ$ be a set with $|\cZ|>n$. If there is an embedding $\phi\colon Q_n\to \QQ(\cZ)$,
then there exist a subset $\cX\subsetneql\cZ$ with $|\cX|=n$, and an embedding $\phi'\colon \QQ(\cX)\to \QQ(\cZ)$ with the same image as $\phi$ such that
 $\phi'(X)\cap \cX=X$ for all $X\subseteq \cX$.
\end{lemma}

In other words, for any copy $Q'$ of $Q_n$ in the hosting lattice $\QQ(\cZ)$ there is some $\cX\subsetneql\cZ$ 
such that the vertex-wise restriction of $Q'$ to $\cX$ is equal to $\QQ(\cX)$.


\section{Upper bound on $R(\cS\cD_{s,t},Q_n)$}\label{sec:non-triv}

\subsection{Counting permutations}

In this subsection we bound the number of permutations with a special property, in preparation for our proof of Theorem \ref{thm:SD}. 
A permutation $\pi\colon [k]\to[k]$ is called \textit{$r$-proper} if for every $j\in[k]$, $|\{\ell\le j : \pi(\ell)\ge j-1\}|\le r$. If $r$ is clear from the context we usually omit the parameter.
For example, the permutation $\hat{\pi}$ given by $(\hat{\pi}(1),\dots,\hat{\pi}(k))=(k,1,3,4,5,\dots,k-1,2)$, see Figure \ref{fig:pi_hat}, is not $1$-proper because at $j=3$, $\{\ell\le 3 : \hat{\pi}(\ell)\ge2\}=\{1,3\}$. However, $\hat{\pi}$ is $2$-proper.

\begin{figure}[h]
\centering
\includegraphics[scale=0.6]{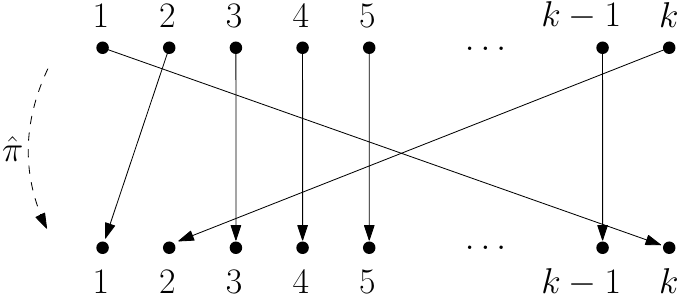}
\caption{The permutation $\hat{\pi}\colon [k]\to[k]$.}
\label{fig:pi_hat}
\end{figure}

\begin{lemma}\label{lem:count_perm}
Let $r,k\in\N$. There are at most $2^{(r+\log r )k}$ distinct $r$-proper permutations $\pi\colon [k]\to[k]$.
\end{lemma}

\begin{proof}
For an $r$-proper permutation $\pi$, we say that an index $i\in[k]$ is \textit{bad} if $\pi(i)\ge i$, and \textit{good} if $\pi(i)\le i-1$.
Let $\cB_\pi$ and $\cG_\pi$ denote the set of indices that are bad and good, respectively, i.e.\ the sets partition $[k]$.
Again considering the example $(\hat{\pi}(1),\dots,\hat{\pi}(k))=(k,1,3,4,5,\dots,k-1,2)$, we have that $\cB_{\hat{\pi}}=\{1\}\cup\{3,4,\dots,k-1\}$ and $\cG_{\hat{\pi}}=\{2,k\}$.

Given a proper permutation $\pi$, the \textit{proper restriction} $\rho$ of $\pi$ is the restriction of $\pi$ to its bad indices, i.e.\ $\rho\colon \cB_\pi\to [k]$ with
$\rho(i)=\pi(i)$ for every $i$. Note that $\rho$ does not depend on $r$.
For example, the proper restriction of $\hat{\pi}$ is $\hat{\rho}\colon [k-1]\setminus\{2\}\to [k]$ with $\hat{\rho}(1)=k$ and $\hat{\rho}(i)=i$ for $3\le i\le k-1$.
Observe that a function $\rho$ can be the proper restriction of distinct proper permutations. Let $\Pi$ be the set of all proper permutations $\pi\colon [k]\to[k]$. 
If a function $\rho$ is the proper restriction of some $\pi\in \Pi$, we say that $\rho$ is a \textit{$\Pi$-restriction}.
To avoid ambiguity, we denote the domain of $\rho$  by $\cB_\rho$.
Inheriting the properties of a proper permutation, $\rho$ is injective and 
$$\big|\{\ell\in\cB_\rho : \ell\le j, \ \rho(\ell)\ge j-1\}\big|\le r.$$

In the following we bound $|\Pi|$ by first estimating $|\{\rho: ~ \rho \text{ is }\Pi\text{-restriction}\}|$,
and then bounding $|\{\pi\in \Pi: ~ \rho\text{ is proper restriction of }\pi\}|$ for every fixed $\rho$.
\\

\textbf{Claim 1:} There are at most $2^{rk}$ distinct $\Pi$-restrictions.\\
\textit{Proof of Claim 1:} 
We show that every $\Pi$-restriction has a distinct representation as a collection of $r$ vectors $V_1, \dots, V_{r}\in\{\zero,\one\}^k$, which implies that there are at most $2^{rk}$ $\Pi$-restrictions.
Let $\rho$ be a $\Pi$-restriction with domain $\cB_\rho$. For every $i\in\cB_\rho$ we define an integer interval $I_i=[i,\rho(i)+1]$.
Consider the \textit{interval graph} $G$ given by intervals $I_i$, 
i.e.\ the graph on vertex set $\cB_\rho$ where $\{i,j\}$ is an edge if and only if $i\neq j$ and $I_i\cap I_j\neq\varnothing$.
In the following we use terminology common in graph theory, for a formal introduction we refer the reader to Diestel~\cite{D}.

Next we bound the maximal size of a clique in $G$.
Suppose that vertices $i_1,\dots,i_m$ form a clique in $G$, then the intervals $I_{i_1},\dots,I_{i_m}$ pairwise intersect.
Thus there exists an integer $j\in[k]$ such that $j\in I_{i_1}\cap \dots \cap I_{i_m}$.
Now 
$$m=\big|\{\ell\in\cB_\rho : ~ j\in I_\ell\}\big|= \big|\{\ell\in\cB_\rho : ~\ell\le j, \ \rho(\ell)+1\ge j\}\big|\le r,$$
where the last inequality holds since $\rho$ is a proper restriction.
Thus there is no clique of size $r+1$ in $G$. It is common knowledge that interval graphs are perfect, 
so there exists a proper vertex coloring of $G$ using at most $r$ colors. Fix such a coloring $c$ of $G$ with set of colors $[r]$. 
Note that for each color class the corresponding intervals are pairwise disjoint.

For every fixed $s\in[r]$, let the set of indices with color $s$ be $\cB_s=\{i\in\cB_\rho: ~ c(I_i)=s\}$.
We define a vector $V_{s}\in\{\zero,\one\}^k$ as follows. 
Let 
$$V_s(i)=\dots=V_s(\rho(i))=\one\text{ for any }i\in\cB_s\quad\text{ and }\quad V_s(j)=\zero\text{ for all other }j\in[k].$$
Since the intervals $I_i$, $i\in\cB_s$, are pairwise disjoint, $V_s$ is well-defined. Moreover, we obtain that $V_s(\rho(i)+1)=\zero$ for every $i\in\cB_s$. 
This implies that $V_s(i-1)=\zero$, if defined, for $i\in\cB_s$.
Observe that the vector $V_s$ encodes all indices in $\cB_s$ and their respective functional values $\rho(i)$, $i\in\cB_s$:
If for some $j\in[k]$, $V_s(j)=\one$ and $V_s(j-1)=\zero$, then $j\in\cB_s$ and $\rho(j)$ is given by the maximal index $j'$ such that $V_s(j)=\dots=V_s(j')=\one$.

We obtain a vector representation $V_1, \dots, V_{r}$ of $\rho$.
It is easy to see that distinct $\Pi$-restrictions have distinct representations.
There are at most $(2^k)^{r}$ distinct such vector respresentations, which proves the claim.
\\

\textbf{Claim 2:} Given a fixed $\Pi$-restriction $\rho$, the number of proper permutations $\pi$ with proper restriction $\rho$ is at most $r^{k}$.\\
\textit{Proof of Claim 2:} 
Let $\rho$ be a fixed $\Pi$-restriction and let $\cG=[k]\setminus \cB_\rho$. 
We count the possible assignments of good indices $i\in\cG$ in a proper permutation given that $\pi(\ell)=\rho(\ell)$ for every $\ell\in\cB_\rho$.
For this purpose we iterate through all good indices $i\in\cG$ in increasing order while counting the choices for each $\pi(i)$. Observe that $1\notin\cG$ since $\pi(1)\ge 1$ for any permutation $\pi$.
Fix an $i\in\cG$, i.e.\ $i\ge 2$. Suppose that all indices $\ell\in\cG\cap[i-1]$ are already assigned to an integer $\pi(\ell)\le \ell-1$ and all $\ell\in\cB_\rho$ are assigned to $\pi(\ell)=\rho(\ell)$.
There are two conditions on the choice of $\pi(i)$:
On the one hand $i$ is a good index, so we require $\pi(i)\in [i-1]$.
On the other hand $\pi$ is injective, thus $\pi(i)\neq \pi(\ell)$ for all $\ell<i$.
Therefore $\pi(i)\in [i-1]\setminus \{\pi(\ell)\in[i-1]: \ell<i\}$. We evaluate the size of this set using the fact that $|\{\ell<i:\pi(\ell)\ge i-1\}|\le r$:
\begin{align*}
\big|\{\pi(\ell)\in[i-1]: \ell<i\}\big|&=\big|\{\ell<i:\pi(\ell)\le i-1\}\big|= \big| [i-1]\setminus\{\ell<i:\pi(\ell)>i-1\}\big| \\
&\ge \big| [i-1]\setminus\{\ell<i:\pi(\ell)\ge i-1\}\big| \ge (i-1)-r.
\end{align*}
Thus also
$$\big| [i-1]\setminus \{\pi(\ell)\in[i-1]: \ell<i\} \big|\le(i-1)-(i-1-r)=r.$$
Hence, there are at most $r$ choices for selecting $\pi(i)$ for each $i\in\cG$. 
Note that $|\cG|\le k$, consequently the number of proper permutations with proper restriction $\rho$ is at most $r^k$. 
\\

\noindent Combining both claims, the number of proper permutations is at most 
$$\sum_{\rho\text{ is }\Pi\text{-restriction}} \big|\{\pi\in\Pi \colon ~ \rho \text{ is proper restriction of }\pi\}\big|\le 2^{rk}r^k=2^{(r+\log r)k}.$$
\end{proof}

The bound provided here is not best possible. With a more careful approach the number $N(k,r)$ of $r$-proper permutations $\pi\colon [k]\to[k]$ can be bounded between
$$r^k\le N(k,r) \le (2r)^{2k}.$$
Studying this extremal function might be of independent interest.

\subsection{Proof of Theorem \ref{thm:SD}}

Before presenting the proof of Theorem \ref{thm:SD} we give a lemma which is purely computational and follows the lines of a similar claim by Gr\'osz, Methuku and Tompkins \cite{GMT}.

\begin{lemma}\label{lem:count_est}
Let $n\in\N$, and let $c\in\R$ be a positive constant. 
Let $k=\frac{(2+\epsilon)n}{\log n}$ where $\epsilon=3(\log\log n +\log e + c +2)(\log n)^{-1}$. Then for sufficiently large $n$,
$$k!>2^{ck}\cdot 2^{2(n+k)}.$$
\end{lemma}
\begin{proof}
By Stirling's formula $k!>\left(\frac{k}{e}\right)^k=2^{k(\log k -\log e)}$.
We shall show that $k(\log k -\log e)>ck+2(n+k)$.
Using the fact that $k=\frac{(2+\epsilon)n}{\log n}$, we obtain
\begin{align*}
&k\big(\log k-\log e-c-2\big)-2n\\
&\ge \frac{(2+\epsilon)n }{\log n}\big(\log (2+\epsilon) + \log n -\log\log n-\log e-c-2\big) -2n\\
&\ge \epsilon n- \frac{n}{\log n}(2+\epsilon)\big(\log\log n +\log e+c+2\big)>0,
\end{align*}
where the last inequality holds for sufficiently large $n$.
\end{proof}

\begin{proof}[Proof of Theorem \ref{thm:SD}]
For any $s\le t$ note that $\cS\cD_{s,t}$ is an induced subposet of $\cS\cD_{t,t}$, so it suffices to show the Ramsey bound for $s=t$.
Let $k=\frac{(2+\epsilon)n}{\log n}$ where 
$\epsilon=3(\log\log n +\log e + c +2)(\log n)^{-1}$ and $c=2t+2+\log (2t+2)$.
Let $\cX$ and $\cY$ be disjoint sets with $|\cX|=n$, $|\cY|=k$. Consider an arbitrary blue/red coloring of $\QQ(\cX\cup\cY)$ with no red copy of $Q_n$.
We shall show that there is a blue copy of $\cS\cD_{t,t}$ in this coloring. 

There are $k!$ linear orderings of $\cY$. For every linear ordering $\tau$ of $\cY$, Lemma \ref{lem:chain} provides a blue $\cY$-chain $C^\tau$ in $\QQ(\cX\cup\cY)$ corresponding to $\tau$, say on vertices $Z^\tau_0\subsetneql Z^\tau_1\subsetneql \dots \subsetneql Z^\tau_k$.
Consider the smallest vertex $Z^\tau_0$ as well as the largest vertex $Z^\tau_k$. 
Both vertices are subsets of $\cX\cup\cY$, so there are $2^{2(n+k)}$ distinct pairs $(Z^\tau_0,Z^\tau_k)$.
By pigeonhole principle there is a collection $\tau_1,\dots,\tau_m$ of $m=\frac{k!}{2^{2(n+k)}}$ distinct linear orderings of $\cY$ such that all of the corresponding $\cY$-chains $C^{\tau_i}$ have both $Z^{\tau_i}_0$ and $Z^{\tau_i}_k$ in common.
Lemma \ref{lem:count_est} shows that $m> 2^{ck}$. 
\\

Fix an arbitrary $\sigma\in\{\tau_1,\dots,\tau_m\}$. By relabelling $\cY$ we can suppose that $\sigma=(1,\dots,k)$, i.e.\ $1<_\sigma \dots <_\sigma k$.
Consider a linear ordering $\tau_j$, $j\in[m]$, (allowing that $\tau_j=\sigma$) and let $\tau_j=(y_1,\dots,y_k)$.
Then we say that $\tau_j$ is \textit{$t$-close} to $\sigma$ for some $t\in\N$ if for every $i\in[k-t]$, either $[i]\subseteq \{y_1,\dots,y_{i+t}\}$ or $\{y_1,\dots,y_i\}\subseteq [i+t]$. 
For example the linear ordering $(4,5,\dots,k,1,2,3)$ is $3$-close to $\sigma$ since the first $i$ elements of this linear ordering are contained in $[i+3]$, for any $i\in[k-3]$.
However, our example is not $2$-close to $\sigma$, because neither $\{1\}\subseteq\{4,5,6\}$ nor $\{4\}\subseteq[3]$.
In the remaining proof we distinguish two cases.
If there is a linear ordering $\tau_j$ which is not $t$-close to $\sigma$, we build a copy of $\cS\cD_{t,t}$ from the $\cY$-chains corresponding to $\sigma$ and $\tau_j$.
If every linear ordering $\tau_1,\dots,\tau_m$ is $t$-close to $\sigma$, we find $m>2^{ck}$ permutations fulfilling the property of Lemma \ref{lem:count_perm}, thus we arrive at a contradiction.
\\

\textbf{Case 1:} There is a linear ordering $\tau\in\{\tau_1,\dots,\tau_m\}$ which is not $t$-close to $\sigma$.\smallskip\\
Suppose that the $\cY$-chains corresponding to $\sigma$ and $\tau$ are given by $Z^{\sigma}_0,\dots,Z^{\sigma}_k$ and  $Z^{\tau}_0,\dots,Z^{\tau}_k$, respectively. 
Recall that $Z^\sigma_0=Z^\tau_0$ and $Z^\sigma_k=Z^\tau_k$.
Since $\tau$ is not $t$-close to $\sigma$ there is an index $i\in[k-t]$ such that neither $[i]\subseteq \{y_1,\dots,y_{i+t}\}$ nor $\{y_1,\dots,y_{i}\}\subseteq [i+t]$.
Note that $i< k-t$ because for $i=k-t$ we have $[k-t]\subseteq[k]=\cY=\{y_1,\dots,y_{k}\}$.
The definition of $\cY$-chains provides that $Z^\sigma_i \cap \cY = [i]$ and $Z^\tau_{i+t} \cap \cY=\{y_1,\dots,y_{i+t}\}$, thus $Z^\sigma_i\not\subseteq Z^\tau_{i+t}$. By transitivity, $Z^\sigma_{j}\not\subseteq Z^\tau_{j'}$ for any two $j,j'\in\{i,\dots,i+t\}$.
Similarly, $Z^\tau_i\not\subseteq Z^\sigma_{i+t}$ and so $Z^\tau_j\not\subseteq Z^\sigma_{j'}$.
This implies that the set $$\cP=\left\{Z^\sigma_j,Z^\tau_j : ~ j\in \{0,k\}\cup\{i,\dots,i+t\}\right\}$$
forms a copy of $\cS\cD_{t,t}$.
Furthermore, every vertex of $\cP$ is included in a blue $\cY$-chain and thus colored blue.
This completes the proof for Case 1.
\\

\textbf{Case 2:} Every linear ordering $\tau\in\{\tau_1,\dots,\tau_m\}$ is $t$-close.\smallskip\\
Here we use the fact that every linear ordering $\tau_j$, $j\in[m]$, is obtained by permuting the linear ordering $\sigma$.
Fix an arbitrary $\tau\in\{\tau_1,\dots,\tau_m\}$, and let $\tau=(y_1,\dots,y_k)$. We say that the permutation \textit{corresponding} to $\tau$ is $\pi\colon [k]\to[k]$ with $\pi(\ell)=y_\ell$.
We show that $\pi$ has the following property.
\\

\textbf{Claim:} For every $j\in[k]$, $|\{\ell\le j : ~ \pi(\ell)> j+t\}|\le t$.\smallskip\\ 
The statement is trivially true if $j+t> k$. Now fix some $j\in[k-t]$.
By $t$-closeness of $\tau$ either $\{\pi(1),\dots,\pi(j)\}=\{y_1,\dots,y_j\}\subseteq [j+t]$ or $[j]\subseteq \{y_1,\dots,y_{j+t}\}=\{\pi(1),\dots,\pi(j+t)\}$.

If $\{\pi(1),\dots,\pi(j)\}\subseteq [j+t]$, then for every $\ell\le j$ we have $\pi(\ell) \le j+t$. 
Therefore $\{\ell\le j : \pi(\ell)> j+t\}=\varnothing$ and the statement holds.

If $[j]\subseteq \{\pi(1),\dots,\pi(j+t)\}$, then let $I=\{\pi(1),\dots,\pi(j+t)\}\setminus [j]$, note that $|I|=t$.
Observe that for every $\ell\le j$ with $\pi(\ell)> j+t$, we know in particular that $\pi(\ell)\notin[j]$, thus $\pi(\ell)\in I$.
Since $\pi$ is bijective, $$\big|\{\ell\le j : ~ \pi(\ell)> j+t\}\big|=\big|\{\pi(\ell): ~ \ell\le j, \ \pi(\ell)> j+t\}\big|\le |I|=t.$$
This proves the claim.
\\

In particular, $\pi$ has the property that $|\{\ell\le j : \pi(\ell)\ge j-1\}|\le 2t+2$ for every $j\in[k]$, i.e.\ $\pi$ is $(2t+2)$-proper.
Note that distinct linear orderings $\tau_i$, $i\in[m]$, correspond to distinct permutations $\pi_i\colon [k]\to[k]$.
Lemma \ref{lem:count_perm} provides that the number of $(2t+2)$-proper permutations $\pi_i$ is at most 
$$m\le 2^{(2t+2+\log (2t+2) )k}=2^{ck}.$$
Recall that by Lemma \ref{lem:count_est}, $m>2^{ck}$, so we arrive at a contradiction.
\end{proof}

\section{Proofs of Theorems \ref{thm_cP_UB} and \ref{lem_CA}} \label{sec:triv}

\begin{proof}[Proof of Theorem \ref{thm_cP_UB}]
First we determine $R(C_{t_1},Q_n)$. The lower bound follows from Theorem \ref{thm-MAIN}, so we only need to show that $R(C_{t_1},Q_n)\le n+t_1-1$.
Suppose that in a blue/red colored Boolean lattice $\QQ^1:=\QQ([n+t_1-1])$ there is no blue copy of $C_{t_1}$. 
Then by Corollary \ref{cor:chain} there is a red copy of $Q_n$ in $\QQ^1$.
\\

Next we consider the poset Ramsey number $R(\CC_{t_1,t_2},Q_n)$. Let $N=n+t_1+1$. 
In order to prove that $R(\CC_{t_1,t_2},Q_n)\le N$, we consider an arbitrarily blue/red colored Boolean lattice $\QQ^2:=\QQ([N])$. 
We show that there is either a red copy of $Q_n$ or a blue copy of $\CC_{t_1,t_2}$ in this coloring.
Corollary \ref{cor:chain} shows that there is either a red copy of $Q_n$ or a blue chain of length $t_1+2$.
If the former happens, the proof is already complete, so suppose that there is a blue chain $C$ on vertices $Z_0\subset Z_1 \subset \dots \subset Z_{t_1+1}$.
Consider the subposet $C'$ of $C$ on vertices $Z_1,\dots,Z_{t_1}$, i.e.\ obtained by discarding the minimum and maximum vertex of $C$.
The poset $C'$ is a chain of length $t_1$.
Note that $Z_1\neq \varnothing$ since it has a proper subset $Z_0\subset Z_1$, thus there is an element $a\in Z_1$.
Similarly, $Z_{t_1}\neq [N]$, so we find a $b\in[N]\setminus Z_{t_1}$. We obtain that $\{a\}\subseteq Z_1\subseteq \dots \subseteq Z_{t_1} \subseteq [N]\setminus\{b\}$.

\begin{figure}[h]
\centering
\includegraphics[scale=0.6]{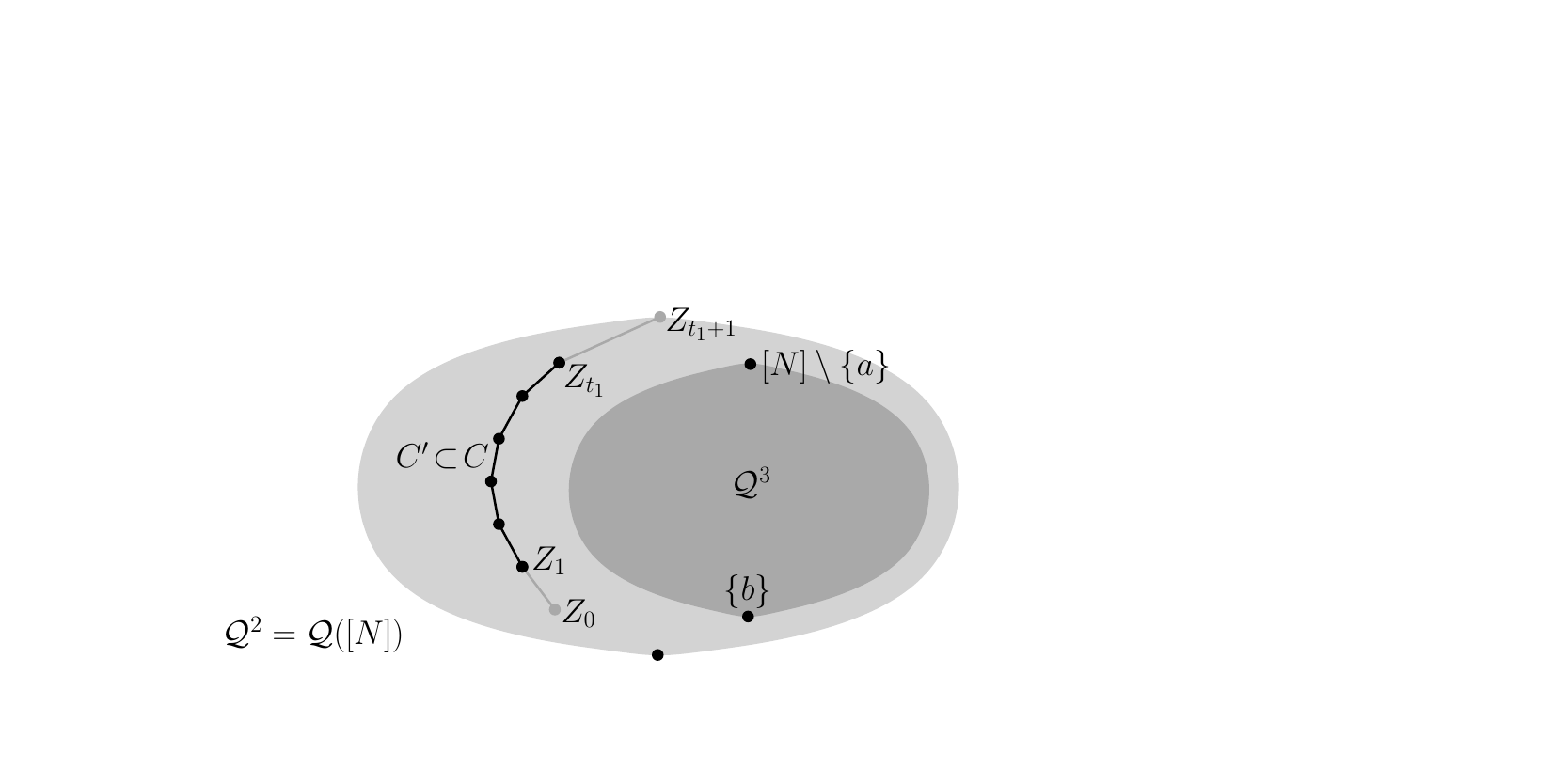}
\caption{$C'$ and $\QQ^3$ are parallel}
\label{fig:CQ3}
\end{figure}
We consider the subposet $\QQ^3:=\{Z\in\QQ^2 \colon b\in Z,\ a\notin Z\}$, see Figure \ref{fig:CQ3}. Note that $\QQ^3$ is a Boolean lattice of dimension $N-2=n+t_1-1$ with a blue/red coloring induced by the coloring of $\QQ^2$.
The posets $\QQ^3$ and $C'$ are parallel, i.e.\ for every two $Z\in\QQ^3$ and $U\in C'$ we have $Z||U$, because $a\in U$, $a\notin Z$ and $b\in Z$, $b\notin U$.
Applying the first bound shown in this proof we obtain that $R(C_{t_2},Q_n)=n+t_2-1\le n+t_1-1$.
Since $\QQ^3$ has dimension $n+t_1-1$, there is either a red copy of $Q_n$ or a blue copy $D$ of $C_{t_2}$ in $\QQ^3$.
In the first case the proof is complete, and in the second case the vertices of $C'$ and $D$ form a blue copy of $\CC_{t_1,t_2}$, which also completes the proof.
Thus, $R(\CC_{t_1,t_2},Q_n)\le N$.
\\

It remains to show that $R(\CC_{t_1,t_2},Q_n)\ge N=n+t_1+1$. 
We verify this lower bound by introducing a layered coloring of $Q_{N-1}$ which neither contains a blue copy of $\CC_{t_1,t_2}$ nor a red copy of $Q_n$.
Consider the Boolean lattice $\QQ^4:=\QQ([N-1])$ which is colored such that
layer $0$ and layer $N-1$ (i.e.\ both one-element layers) are monochromatically blue, $t_1-1$ arbitrary additional layers are blue,
and all remaining $N-(t_1+1)=n$ layers are red.

Since $Q_n$ has height $n+1$, there is no monochromatic red copy of $Q_n$ in this coloring.
Assume that there is a blue copy $\cP$ of $\CC_{t_1,t_2}$ in~$\QQ^4$.
Note that $\varnothing\notin\cP$ and $[N-1]\notin\cP$, 
because each of these two vertices is comparable to all other vertices of $\QQ^4$ and in $\cP$ no vertex comparable to all other vertices of $\cP$.
Other than $\varnothing$ and $[N-1]$, $\QQ^4$ has only $t_1-1$ layers containing blue vertices, but $h(\cP)=h(\CC_{t_1,t_2})=t_1$, so there can not be a blue copy $\cP$ of $\CC_{t_1,t_2}$.
\end{proof}


\begin{proof}[Proof of Theorem \ref{lem_CA}]
Let $N=n+t+2$. In order to show the upper bound on $R(\CC_{t,t-1,t'},Q_n)$, fix an arbitrary blue/red coloring of $\QQ^1:=\QQ([N])$.
We shall find either a red copy of $Q_n$ or a blue copy of $\CC_{t,t-1,t'}$ in $\QQ^1$.
By Corollary \ref{cor:chain} we find either a red copy of $Q_n$ or a blue chain $C$ of length $t+2$.
In the first case the proof is complete, so assume the second case. 
Let $C$ have vertices $Z_0\subset \dots \subset Z_{t+1}$.
Consider the subposet $C'$ of $C$ on vertices $Z_1,\dots,Z_{t}$, which is a chain of length $t$.
Note that $Z_1\neq\varnothing$ and $Z_t\neq [N]$, thus we find some $a,b\in[N]$ such that
$\{a\}\subseteq Z_1\subseteq \dots \subseteq Z_{t} \subseteq [N]\setminus\{b\}$.

Now consider the subposet $\QQ^2:=\{Z\in\QQ^1 \colon b\in Z,\ a\notin Z\}$, which is a Boolean lattice of dimension $N-2=n+t$ with a blue/red coloring induced by $\QQ^1$. 
Theorem \ref{thm_cP_UB} yields that $\QQ^2$ contains either a red copy of $Q_n$ or a blue copy $\cP$ of $\CC_{t-1,t'}$.
In the first case we found a red copy of $Q_n$ in $\QQ^1$. 
In the second case we know for every two $Z\in \cP$ and $U\in C'$ that $a\in U$, $a\notin Z$, $b\notin U$ and $b\in Z$, so $Z||U$.
Consequently, the vertices of $\cP$ and $C'$ induce a blue copy of $\CC_{t,t-1,t'}$.
\\

It remains to verify the lower bound.
For this purpose we shall find a blue/red coloring of $\QQ^3:=\QQ([N-1])$ which neither contains a red copy of $Q_n$ nor a blue copy of $\CC_{t,t-1,t'}$.
Note that $t\ge t'+1\ge 2$.
Color all vertices in the four layers $\ell\in\{0,1,N-2,N-1\}$ monochromatically in blue. Color $t-2$ arbitrary additional layers blue and all remaining $N-(4+t-2)=n$ layers red. 
Clearly, this coloring contains no red copy of $Q_n$ since $h(Q_n)=n+1$. 
Assume for a contradiction that there is a blue copy $\cP$ of $\CC_{t,t-1,t'}$ in $\QQ^3$. 
In~$\cP$ denote the blue chain of length $t$ by $C$, see Figure \ref{fig:CDF}. Furthermore there is a chain $D$ of length $t-1$ in $\cP$ which is parallel to $C$. 
Let $F$ be a vertex of $\cP$ which is neither in $C$ nor in $D$, i.e.\ $F$ is incomparable to every vertex in $C$ or $D$.

\begin{figure}[h]
\centering
\includegraphics[scale=0.6]{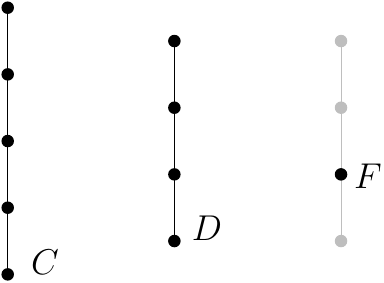}
\caption{Chains $C$ and $D$ and vertex $F$ in a copy of $\CC_{t,t-1,t'}$}
\label{fig:CDF}
\end{figure}

It is easy to see that neither $\varnothing$ nor $[N-1]$ are in $\cP$, because both of these vertices are comparable to every other vertex in $\QQ^3$ which does not occur in $\cP$.
Excluding the two vertices $\varnothing$ and $[N-1]$, there are exactly $t$ layers containing blue vertices, including layer $1$ and layer $N-2$.

Recall that $C$ is a blue chain of length $t$, therefore the minimum vertex $Z_1$ of $C$ is in layer~$1$, while the maximum vertex $Z_t$ of $C$ is in layer~$N-2$.
Thus we find $a,b\in[N-1]$ such that $Z_1=\{a\}$ and $Z_t=[N-1]\backslash\{b\}$.
Let $\QQ^4:=\{Z\in\QQ^3 \colon b\in Z ,\ a\notin Z\}$, this poset is a Boolean lattice.
Since $F$ is incomparable to every vertex in $C$, we obtain that $a\notin F$ and $b\in F$. Thus $F\in\QQ^4$ and in particular $F\supseteq \{b\}$ and $F\subseteq [N-1]\backslash \{a\}$.
Similarly, $D\subseteq\QQ^4$.

The coloring of $\QQ^3$ induces a layered coloring of $\QQ^4$ where exactly $t$ of the layers are blue.
Two of these blue layers in $\QQ^4$ are the one-element layers given by $\{b\}$ and $[N-1]\setminus\{a\}$.
Since the chain $D$ has height $h(D)=t-1$, either $\{b\}\in D$ or $[N-1]\backslash \{a\}\in D$. 
This is a contradiction, because $F$ is incomparable to every vertex in $D$ but both $\{b\}$ and $[N-1]\backslash \{a\}$ are comparable to $F$.
\end{proof}


\section{Bounds on $R(A_t,Q_n)$}\label{sec:antichain}

\subsection{Erd\H{o}s-Szekeres variant}
In preparation for the proof of Theorem \ref{thm_antichain} we reshape the following well-known result of Erd\H{o}s and Szekeres \cite{ES}.

\begin{theorem}[Erd\H{o}s-Szekeres \cite{ES}]\label{erdos-szekeres}
Let $m\in\N$. For an $(m^2+1)$-element set $\cZ$, let $S=(a_1,a_2,\dots,a_{m^2+1})$ be a sequence consisting of distinct elements of $\cZ$.
Let $\tau$ be an arbitrary linear ordering of $\cZ=\{a_1,\dots,a_{m^2+1}\}$.
Then there exists a subsequence $(a_{i_1},\dots,a_{i_{m+1}})$ of $S$ of length $m+1$ such that \quad
$a_{i_1}<_\tau \dots <_\tau a_{i_{m+1}}\quad\text{or}\quad a_{i_{m+1}}<_\tau \dots <_\tau a_{i_1}.$
\end{theorem}

In the following we refer to a finite sequence of distinct elements as a \textit{permutation sequence} on the set of these elements.
We say that a sequence $(b_1,b_2,\dots,b_\ell)$ is an \textit{undirected subsequence} of some sequence $S$ 
if either $(b_1,b_2,\dots,b_\ell)$ or $(b_\ell,b_{\ell-1},\dots,b_1)$ is a subsequence of $S$. 
Let $S^1,\dots,S^m$, $m\in\N$, be a collection of permutation sequences on the same set. 
If $(b_1,b_2,\dots,b_\ell)$ is an undirected subsequence of every $S^i$, it is referred to as a \textit{common undirected subsequence} of $S^1,\dots,S^m$.

\begin{corollary}\label{ac_sequence}\label{cor_es}
Let $S$ and $T$ be two permutation sequences on an $(m^2+1)$-element set $\cZ$.
Then there exists a common undirected subsequence of $S$ and $T$ which has length $m+1$.
\end{corollary}
\begin{proof}
Let $S=(a_1,a_2,\dots,a_{m^2+1})$, i.e.\ $\cZ=\{a_1,\dots,a_{m^2+1}\}$. For $\ell\in[m^2+1]$, let $j_\ell$ be indices such that $T=(a_{j_1},\dots,a_{j_{m^2+1}})$.
Consider the linear ordering $\tau$ given by $a_{j_1}<_\tau \dots <_\tau a_{j_{m^2+1}}$.
Then the Erd\H{o}s-Szekeres theorem, Theorem \ref{erdos-szekeres}, yields a subsequence $(a_{i_1},\dots,a_{i_{m+1}})$ of $S$ which is also an undirected subsequence of $T$.
In particular, $(a_{i_1},\dots,a_{i_{m+1}})$ is a common undirected subsequence of $S$ and $T$.
\end{proof}

By iteratively applying Corollary \ref{cor_es} we obtain a lemma that we need later on.

\begin{lemma}\label{ac_ordering}
Let $d\in\N$ and $N\ge 2^{2^{d-1}}+1$. Let $\cZ$ be an $N$-element set.
Let $\tau_1,\dots,\tau_d$ be arbitrary linear orderings of $\cZ$. Then there exist pairwise distinct $x,y,z\in\cZ$ such that for every $i\in[d]$,
$$x <_{\tau_i} y <_{\tau_i} z\qquad \text{ or }\qquad z<_{\tau_i}y<_{\tau_i} x.$$
\end{lemma}
\begin{proof}
For each $i\in[d]$, say that $\tau_i$ is given by $a^i_1<_{\tau_i} a^i_2<_{\tau_i}\dots <_{\tau_i} a^i_N$.
Then let $S(\tau_i)$ be the sequence $(a^i_1, a^i_2,\dots , a^i_N)$. Note that $S(\tau_i)$ is a permutation sequence on $\cZ$.
We show that there is a common undirected subsequence of $S(\tau_1),\dots,S(\tau_d)$ of length~$3$, 
because for such a subsequence $(x,y,z)$ either 
$x <_{\tau_i} y <_{\tau_i} z$ or $z<_{\tau_i}y<_{\tau_i} x$ for every $i\in[d]$.
We proceed with an iterative argument.
\\

Let $T^1=S(\tau_1)$; note that $|T^1|\ge2^{2^{d-1}}+1$. 
For $i\in[d-1]$, suppose that $T^i$ is a common undirected subsequence of all $S(\tau_j)$, $j\in[i]$, and has length at least $2^{2^{d-i}}+1$.
Let $\cZ^i$ be the underlying set of $T^i$, and let $S^{i}$ be the restriction of $S(\tau_{i+1})$ to $\cZ^i$.
Then both $T^i$ and $S^{i}$ are permutation sequences on $\cZ^i$.
By Corollary \ref{ac_sequence}, there is a common undirected subsequence $T^{i+1}$ of $T^i$ and $S^{i}$ 
of length at least $(2^{2^{d-i}})^{\frac12}+1=2^{2^{d-(i+1)}}+1$.
Since $T^{i+1}$ is an undirected subsequence of $T^i$, $T^{i+1}$ is also an undirected subsequence of every $S(\tau_j)$, $j\in[i]$.
Furthermore, because $T^{i+1}$ is an undirected subsequence of $S^{i}$, it is also an undirected subsequence of $S(\tau_{i+1})$.

After $d-1$ steps, we obtain a sequence $T^d$ of length at least $2^{2^{0}}+1=3$ which is a common undirected sequence of all $S^j$, $j\in[d]$.
Choose an arbitrary $3$-element subsequence $(x,y,z)$ of $T^d$. Then $x$, $y$ and $z$ have the desired properties.
\end{proof}

We remark that the bound $N\ge 2^{2^{d-1}}+1$ in Lemma \ref{ac_ordering} is tight: It is widely known that there is a sequence of $m^2$ distinct elements which does not meet the property in the Erd\H{o}s-Szekeres Theorem. Given such a sequence we can straightforwardly construct a collection of $d$ linear orderings of a $2^{2^{d-1}}$-element set such that no triple $x,y,z$ has the property of Lemma \ref{ac_ordering}.

\subsection{Proof of Theorem \ref{thm_antichain}}\label{sec_ac_small}

\begin{proof}[Proof of Theorem \ref{thm_antichain}]
 Note that $R(A_{3},Q_n)\le R(A_{t},Q_n)\le R(A_{\log \log n},Q_n)$ for all $3\le t\le \log \log n$, 
so it suffices to show that $R(A_3,Q_n)\ge n+3$ and $R(A_{\log\log n},Q_n)\le n+3$. Let $N=n+3$.
\\

For the lower bound, we consider the following coloring of the Boolean lattice $\QQ=\QQ([N-1])$: For all $i\in[N-1]$, color $[i]$ and $[N-1]\backslash[i]$ in blue and color all remaining vertices in red.
Observe that all blue vertices can be covered by the two chains $\big\{[i]\colon ~ i\in[N-1]\big\}$ and $\big\{[N-1]\backslash[i]\colon ~ i\in[N-1]\big\}$, 
thus among any three distinct blue vertices we find two vertices contained in the same chain, i.e.\ there is no blue copy of $A_3$ in $\QQ$.
Now, assume towards contradiction that there is a red copy of $Q_n$ in $\QQ$.
Lemma~\ref{embed_lem} provides that there is a set $\cX\subset [N-1]$, $|\cX|=n$, and an embedding $\phi\colon \QQ(\cX)\to \QQ$ 
such that the image of $\phi$ is monochromatic red and $\phi(X)\cap \cX=X$ for all $X\subseteq\cX$.
Note that $\phi(\varnothing)\neq \varnothing$ as $\varnothing$ is colored blue in $\QQ$, so say that $a\in\phi(\varnothing)$. 
Then $a\notin\cX$ since $\phi(\varnothing)\cap\cX=\varnothing$.
Similarly, we know $\phi(\cX)\neq[N-1]$, so we find an element $b\in[N-1]\setminus\phi(\cX)$ with $b\notin\cX$.
Observe that $a\in \phi(\varnothing)\subseteq\phi(\cX)$, but $b\notin\phi(\cX)$, hence $a\neq b$.
Recall that $|\cX|=n=(N-1)-2$, thus $\cX=[N-1]\backslash\{a,b\}$. 
Because $\phi$ is an embedding of $\QQ(\cX)$ and every vertex of $\QQ(\cX)$ is comparable to $\varnothing$ and $\cX$, 
we obtain that $a\in\phi(X)$ and $b\notin\phi(X)$ for every $X\subseteq\cX$. 
Now the fact that $\phi(X)\cap \cX=X$ implies $\phi(X)=X\cup\{a\}$ for all $X\subseteq\cX$.

If $b\in\{a+1,\dots,N-1\}$, then $[a-1]\subseteq\cX$.
Thus $\phi([a-1])=[a-1]\cup\{a\}=[a]$, but this vertex is colored blue in $\QQ$. This is a contradiction to the fact that $\phi$ has a monochromatic red image.

If $b\in\{1,\dots,a-1\}$, then $[N-1]\backslash [a]\subseteq\cX$. Now $\phi([N-1]\backslash [a])=([N-1]\backslash [a])\cup\{a\}=[N-1]\backslash [a-1]$ which is colored blue in $\QQ$, so again we reach a contradiction.
 \\
 
For the upper bound, let $\cZ$ be an arbitrary $N$-element set. 
Consider an arbitrary coloring of the Boolean lattice $\QQ(\cZ)$ which contains no blue copy of $A_t$ where $t\le \log\log n$.
We shall show that there is a red copy of $Q_n$ in $\QQ(\cZ)$.
\\

By Dilworth's theorem, Theorem \ref{dilworths}, there is a collection of $t-1$ chains $C_1,\dots, C_{t-1}$ which cover all blue vertices.
Without loss of generality, we assume that every $C_i$ is a full chain, i.e.\ a chain on $N+1$ vertices of the form $\varnothing, \{a_1\},\{a_1,a_2\},\dots, \{a_1,\dots,a_N\}=\cZ$.
Then we say that each $C_i$, $i\in[t-1]$, \textit{corresponds} to the unique linear ordering $\tau_i$ of $\cZ$ given by $a_1<_{\tau_i}a_2<_{\tau_i}\dots<_{\tau_i}a_N$. 
Now we apply Lemma \ref{ac_ordering} to the collection of linear orderings $\tau_i$, $i\in [t-1]$, and 
obtain three distinct elements $x,y,z\in\cZ$ such that for every $i\in[t-1]$, $x <_{\tau_i} y <_{\tau_i} z$ or $z<_{\tau_i}y<_{\tau_i} x$.
In particular, if for some blue vertex $Z\in\QQ(\cZ)$ we have $x\in Z$ and $z\in Z$, then also $y\in Z$.
This is because $Z$ is covered by chain $C_j$ for some $j\in[t-1]$ and in the corresponding linear ordering $\tau_j$ either $y<_{\tau_j} z$ or $y<_{\tau_j} x$.

Now assume towards a contradiction that there is no red copy of $Q_n$ in $\QQ(\cZ)$, then let $\cY=\{x,y,z\}$ and let $\cX=\cZ\backslash\cY$. 
Let $\tau$ be the linear ordering of $\cY$ given by $x<_{\tau}z<_\tau y$. 
Applying Lemma \ref{chain_lem} we obtain that there exists a chain containing a blue vertex of the form $X_2\cup\{x,z\}$ for some $X_2\subseteq\cX$.
Then $y\notin X_2$, so this blue vertex contains both $x$ and $z$ but does not contain $y$, which is a contradiction.
\end{proof}

\subsection{Proofs of Theorem \ref{thm_antichain2} and Corollary \ref{cor_ac_asym}}\label{sec_ac_large}
If $t$ is large in terms of $n$, an improved lower bound on $R(A_t,Q_n)$ holds, which is shown using a layered construction, i.e.\ a coloring of the hosting lattice where each layer is colored monochromatically.

\begin{proof}[Proof of Theorem \ref{thm_antichain2}]
Let $n,r,t\in\N$ with $t> \binom{n+2r+1}{r}$.
Let $\QQ=\QQ(\cZ)$ be the Boolean lattice on an arbitrary ground set $\cZ$ on $n+2r+1$ elements.
Consider the following layered blue/red coloring of $\QQ$:
Color all vertices $X\in\QQ$ such that $|X|\le r$ or $|X|\ge n+r+1$ in blue and all other vertices in red.
We see that $\QQ$ consists of $n+2r+2$ layers of which $2r+2$ are colored monochromatically blue and all remaining $n$ layers are colored monochromatically red.
Since $h(Q_n)=n+1$, there is no red copy of $Q_n$ in this coloring. It remains to show that there is no blue copy of $A_t$ in $\QQ$.
\\

We fix a symmetric chain decomposition of $\QQ$ as provided by Theorem \ref{thm_scd}. Let $\CC$ be the family of only those symmetric chains which contain a vertex of size exactly $r$.
For any vertex $X$ of size at most $r$ or at least $n+r+1$, there is some chain $C(X)$ in the decomposition which covers $X$. 
Using the properties of a symmetric chain we obtain $C(X)\in\CC$, thus the chains in $\CC$ cover all vertices of size at most $r$ or at least $n+r+1$.
Thus all blue vertices are covered by chains in $\CC$, whereat $|\CC|=\binom{n+2r+1}{r}<t$. Then Dilworth's theorem implies that there is no blue copy of $A_t$ in $\QQ$.
\end{proof}

\begin{proof}[Proof of Corollary \ref{cor_ac_asym}]
The upper bound is obtained from the general bound by Axenovich and Walzer \cite{AW} and $(\ast)$.
For given $n\ge 3$ and $t\ge 2$, let $r$ be the largest non-negative integer with $t> \binom{n+2r+1}{r}$.
In Theorem \ref{thm_antichain2} we showed that $R(A_t,Q_n)\ge n+2r+2$.
Now we bound $r$ in terms of $n$ and $t$.
Note that by the maximality of $r$,
$$t\le \binom{n+2r+3}{r+1}\le \left(\frac{e(n+2r+3)}{r+1}\right)^{r+1}\le \left(\frac{en}{r+1}+3e\right)^{r+1}\le (2en)^{r+1}.$$
Thus, $r+1\ge \frac{\log t}{\log(2e)+\log n}\ge \frac{\log t}{3+\log n}$.
\end{proof}


\section{Summary of known approaches} \label{sec:survey}

\subsection{Proof techniques}

In order to give an insight on research of off-diagonal poset Ramsey numbers, we briefly survey known proof methods for bounding $R(P,Q_n)$ when $P$ is fixed.
For some easy-to-analyze posets, e.g.\ several trivial posets, no advanced tools are required to get an exact bound. 
An example for this is our proof of Theorem \ref{thm:CC} in Section \ref{sec:triv}, which only relies on the basic observation Corollary \ref{cor:chain}.
Many other posets require more involved proof techniques, and there are three methods which provides \textbf{upper bounds} on $R(P,Q_n)$.

A \textit{blob approach} is used in particular for Boolean lattices $P=Q_m$.
The idea behind this method is to consider a blue/red colored hosting lattice in which we define many \textit{blobs}, i.e.\ small sublattices that are pairwise disjoint, arranged in a product structure.
If any blob is monochromatically blue, we obtain a blue copy of $P$. Otherwise, we find a red copy of $Q_n$ by choosing one red vertex in each blob.
This proof technique was introduced by Axenovich and Walzer \cite{AW}, for a more refined version see Lu and Thompson \cite{LT}.

Gr\'osz, Methuku and Tompkins \cite{GMT} introduced another proof technique, the \textit{chain approach}. 
Here we consider the large number of blue chains obtained by the Chain Lemma, Lemma \ref{lem:chain}, 
and use counting arguments to force the existence of a monochromatically blue copy of $P$.
Exemplary for this method are the proof of Theorem \ref{thm:SD} in Section \ref{sec:non-triv} and Theorem~1 in Winter \cite{QnK}.

A more technical proof method is the \textit{blocker approach}. It tries to strengthen the Chain Lemma in order to get a precise picture on how the blue chains in the coloring are located relative to each other, in fact they are grouped into structures called \textit{blockers}. This proof technique is described in another part of this series, see Axenovich and the author \cite{QnN}.
\\

A general \textbf{lower bound} on $R(P,Q_n)$ for every $P$ is obtained from a trivial layered coloring, see Theorem \ref{thm-MAIN}.
For bounding $R(P,Q_n)$ from below, most constructions slightly refine a layered coloring according to the fixed $P$, see e.g.\ the proofs of Theorems \ref{thm:CC} and \ref{lem_CA} or Gr\'osz, Methuku and Tompkins \cite{GMT}.
Most known proofs for the lower bound strengthen the trivial lower bound just by a constant. 
The only non-marginal improvement of the trivial lower bound is given by Axenovich and the author \cite{QnV} building on the structural insights obtained from the \textit{blocker approach}.

\subsection{Open problems}
In our series of papers we discussed the asymptotic behaviour in the off-diagonal poset Ramsey setting $R(P,Q_n)$ for $P$ fixed and $n$ large.
We close this study by collecting some open problems.
For trivial $P$, i.e.\ posets containing neither $\cV$ nor $\cLa$, Theorem \ref{thm-MAIN} bounds $R(P,Q_n)=n+\Theta(1)$ tight up to an additive constant.
We improved the bounds on this constant in Corollary \ref{cor:LB}, but a general exact bound remains to be determined.
For non-trivial $P$ the picture is more unclear. Theorem \ref{thm-MAIN} provides that 
$n+\tfrac{1}{15}\tfrac{n}{\log n} \le R(P,Q_n) \le c_P \cdot n.$
Axenovich and the author~\cite{QnV} conjectured that the true value of $R(P,Q_n)$ is closer to the lower bound.

\begin{conjecture}[\cite{QnV}]\label{conj:1}
For every fixed poset $P$, $R(P,Q_n)=n + o(n)$. 
\end{conjecture}

The lower bound $R(P,Q_n)=n+\Omega(\frac{n}{\log n})$ is known to be asymptotically tight in the two leading additive terms for some non-trivial $P$, i.e.\ $R(P,Q_n)=n+\Theta(\frac{n}{\log n})$. 
We say that such a poset $P$ is \textit{modest}. Note that $\cV$ and $\cLa$ are modest, and every non-trivial poset contains either $\cV$ or $\cLa$ as a subposet.
Belonging to the class of modest posets are e.g.\ subdivided diamonds, see Theorem \ref{thm:SD}, complete multipartite posets \cite{QnK} and the N-shaped poset $\cN$ \cite{QnN}. 
Notably it remains open whether there exists a non-trivial poset which is \textit{not} modest.
\begin{conjecture}\label{conj:1b}
There is a fixed poset $P$ with $R(P,Q_n)=n + \omega\left(\frac{n}{\log n}\right).$
\end{conjecture}

Known modest posets differ in various poset parameters, for example $\cS\cD_{t,t}$ has large height and $K_{1,t}$ has large width.
However, every known modest poset has order dimension~$2$.
The \textit{order dimension} of $P$ is the minimal number of linear orderings such that $P$ is the intersection of these linear orderings.
Natural candidates for proving Conjecture \ref{conj:1b} are either $Q_3$ or the standard example $S_3$, the $6$-element poset induced by layer $1$ and $2$ of $Q_3$. Both posets have order dimension $3$. 
\\

Determining $R(Q_m,Q_n)$ for $m\in\N$ is one of the most interesting open problems regarding poset Ramsey numbers.
While well-understood for $m\in\{1,2\}$, the asymptotic behaviour of $R(Q_m,Q_n)$ for fixed $m\ge 3$ is only bounded up to a constant linear factor, see Lu and Thompson \cite{LT}.
Conjectures \ref{conj:1} and \ref{conj:1b} are equivalent to the following.
\begin{conjecture}\label{conj:2}
For every fixed $m\in\N$, $R(Q_m,Q_n)=n + o(n)$. Furthermore there is a fixed integer $m\in\N$ with $R(Q_m,Q_n)=n + \omega\left(\frac{n}{\log n}\right).$
\end{conjecture}
\noindent Note that Conjecture \ref{conj:2} is a strengthening of an open conjecture raised by Lu and Thompson~\cite{LT}.
\\


\noindent \textbf{Acknowledgments:}~\quad   Research was partially supported by DFG grant FKZ AX 93/2-1.
The author would like to thank Maria Axenovich for many helpful comments and discussions, and her contribution to other parts of this series.
The author thanks Torsten Ueckerdt for his comments leading to a cleaner argument for Lemma \ref{lem:count_perm}, Felix Clemen for valuable comments on the manuscript as well as the two anonymous referees of a previous version of this manuscript for their careful reading and their constructive feedback.


\end{document}